\documentclass[11pt]{amsart}
\usepackage[dvipsnames]{xcolor}
\usepackage{amsmath, amscd, amssymb, mathrsfs}
\usepackage[all]{xy}
\usepackage{setspace}

\usepackage[]{todonotes} 

\usepackage{marginnote}  
\usepackage{enumitem} 
\usepackage[colorlinks=true]{hyperref}  
\usepackage[right,mathlines,running]{lineno}
\usepackage[abbrev]{amsrefs}

\newtheorem{newthm}{Theorem}

\newtheorem{thm}{Theorem}[section]
\newtheorem{cor}[thm]{Corollary}
\newtheorem{lem}[thm]{Lemma}
\newtheorem{prop}[thm]{Proposition}
\newtheorem{hypothesis}[thm]{Hypothesis} 
\theoremstyle{definition}
\newtheorem{definition}[thm]{Definition}
\theoremstyle{remark}

\newtheorem{rem}[thm]{Remark}
\numberwithin{equation}{section}

\newcommand{\ai}{\sqrt{-1}} 

\newcommand{\Vol}{\mathrm{Vol}}

\newcommand{\rk}{\mathrm{rk}}
\newcommand{\tr}{\mathrm{tr}}

\newcommand{\Id}{\mathrm{Id}}

\newcommand{\surj}{\to\kern-1.8ex\to}

\newcommand{\vertiii}[1]{{\left\vert\kern-0.25ex\left\vert\kern-0.25ex\left\vert #1 
    \right\vert\kern-0.25ex\right\vert\kern-0.25ex\right\vert}}
\allowdisplaybreaks[4]

\begin{document}
\title[Hermitian--Einstein metrics]{A variational approach to the Hermitian--Einstein metrics and the Quot-scheme limit of Fubini--Study metrics} 
\author{Yoshinori Hashimoto and Julien Keller}

\begin{abstract}
	This is a sequel of our paper \cite{HK1} on the Quot-scheme limit and variational properties of Donaldson's functional, which established its coercivity for slope stable holomorphic vector bundles over smooth projective varieties. Assuming that the coercivity is uniform in a certain sense, we provide a new proof of the Donaldson--Uhlenbeck--Yau theorem, in such a way that the analysis involved in the proof is elementary except for the asymptotic expansion of the Bergman kernel.
\end{abstract}

\maketitle

\section*{Introduction}

This is a sequel to our paper \cite{HK1} on the Quot-scheme limit and variational properties of Donaldson's functional $\mathcal{M}^{Don}$. Let $\mathcal{E}$ be a holomorphic vector bundle over a smooth projective variety $X$ of rank $r>1$. In \cite{HK1}, we established a direct relationship between the slope stability of $\mathcal{E}$ and the variational properties of $\mathcal{M}^{Don}$. It is natural to ask whether this result can be used to give a new proof of the correspondence\footnote{This theorem, having a long and rich history, is often called the Kobayashi--Hitchin correspondence in the literature.} between the slope stability of $\mathcal{E}$ and the existence of Hermitian--Einstein metrics on $\mathcal{E}$. The ``easy'' direction of Hermitian--Einstein metrics implying slope stability was established in \cite[Section 7]{HK1}, providing a more geometric point of view of the theorem by Kobayashi \cite{Kobookjp} and L\"ubke \cite{Luebke}. In this paper, assuming a certain strengthening of \cite[Theorem 1]{HK1}, stated as Hypothesis \ref{conjpczero}, we give a new proof of the ``hard'' direction, called the Donaldson--Uhlenbeck--Yau theorem, which states that the slope stability of $\mathcal{E}$ implies the existence of Hermitian--Einstein metrics on $\mathcal{E}$, with relatively elementary analysis except for the asymptotic expansion of the Bergman kernel.

\begin{newthm} \label{mainthmhk}
	Suppose that Hypothesis \ref{conjpczero} is true. Then, we can prove that there exists a Hermitian--Einstein metric on $\mathcal{E}$ if it is slope polystable, in such a way that the analysis involved in the proof is elementary except for the asymptotic expansion of the Bergman kernel.
\end{newthm}

In spite of the drawback of having to assume Hypothesis \ref{conjpczero}, our method has the novelty of relying much less on analysis (in particular nonlinear PDE theory), compared to the original proof by Donaldson \cite{Do-1983,Don1985,Don1987} and Uhlenbeck--Yau \cite{U-Y,U-Y2}. Our proof relies on the formalism involving the Quot-schemes in algebraic geometry, as developed in \cite{HK1}, which consequently reduces the input from hard analysis as stated in Theorem \ref{mainthmhk}. In terms of the analytic results that we do use, i.e.~the Bergman kernel expansion, it seems worth pointing out that we do not need the full extent of the asymptotic expansion of the Bergman kernel (Theorem \ref{thmbergexp}) and we only need one particular consequence of it: the set of Fubini--Study metrics is dense in the space of hermitian metrics on $\mathcal{E}$ (Corollary \ref{corthmbergexp}). While this consequence itself is widely used in K\"ahler geometry, no proof is known that is not based on Theorem \ref{thmbergexp} which is essentially a theorem in analysis. In any case, this is the only advanced analytic tool that we shall rely on in this paper.\\

It is also important to note that the proof of Theorem \ref{mainthmhk} is based on the formulation of the slope stability as \textit{uniform} stability, established in Proposition \ref{uniflemmna}; if $\mathcal{E}$ is slope stable, we shall show that the non-Archimedean Donaldson functional $\mathcal{M}^{\mathrm{NA}} (\zeta ,k)$ introduced in \cite{HK1} satisfies $$\mathcal{M}^{\mathrm{NA}} (\zeta ,k) \geq 2c_{\mathcal{E}}  \cdot J^{\mathrm{NA}} (\zeta , k )$$ for some constant $c_{\mathcal{E}} >0$ and a quantity $J^{\mathrm{NA}} (\zeta , k)$ which plays the role analogous to the non-Archimedean $J$-functional introduced in \cite{BHJ1} (and also, equivalently, in \cite{Dertwisted}).

Recall that the notion of uniform stability played a significant role in the study of constant scalar curvature K\"ahler metrics and K\"ahler--Einstein metrics in e.g.~\cite{BHJ2,BHJ1,Dertwisted}, amongst many others. In particular, Berman--Boucksom--Jonsson \cite{BBJ} proved that a Fano manifold admits a K\"ahler--Einstein metric if it is uniformly $K$-stable. This paper, together with its prequel \cite{HK1}, could be regarded as providing a new framework of the Donaldson--Uhlenbeck--Yau theorem by constructing various vector-bundle analogues of notions that proved useful in the study of constant scalar curvature K\"ahler metrics and K\"ahler--Einstein metrics. This point will be explained further in our forthcoming paper so that the analogies will become clearer.\\

Hypothesis \ref{conjpczero}, which we need to assume, implies in particular that Donaldson's functional is bounded from below when $\mathcal{E}$ is slope (semi)stable (Proposition \ref{pcorrdonna}). Indeed, from a variational point of view, finding a sufficient condition for Donaldson's functional bounded from below is of significant importance. In the appendix, we shall show that Donaldson's functional is bounded from below if we have a quantitative $C^0$-estimate for the hermitian metrics (Theorem \ref{propUnif}), in the sense of \textit{$\delta$-boundedness} as defined in Definition \ref{defdelbdd}. Although we cannot apply this result to our proof of Theorem \ref{mainthmhk} (see Remark \ref{rprpunifqsl}), we believe that it is of independent interest. It is tempting to point out an analogy with the situation in  K\"ahler--Einstein metrics, in which proving the $C^0$-estimate along the continuity path for solving the Monge-Amp\`ere equation has been known to be of crucial importance (see e.g.~\cite[Chapter 6]{TianBook}).
\\

{\bf Outline of the proof of Theorem \ref{mainthmhk}.} The key result is the inequality proved in \cite[Theorem 1]{HK1}, which describes the asymptotic behaviour of Donaldson's functional $\mathcal{M}^{Don}$ in terms of the algebro-geometric quantity $\mathcal{M}^{\mathrm{NA}}$ that involves slopes of $\mathcal{E}$ and its subsheaves (see also Theorem \ref{thmlnsdf}). Hypothesis \ref{conjpczero} is an improved version of this inequality, which implies in particular that Donaldson's functional is bounded from below if $\mathcal{E}$ is slope (semi)stable. We thus take a sequence $\{ h_i \}_{i \in \mathbb{N}} \subset \mathcal{H}_{\infty}$ such that $\mathcal{M}^{Don} (h_i)$ converges to $\inf_{\mathcal{H}_{\infty}}\mathcal{M}^{Don}$. We approximate each $h_i$ by a Fubini--Study metric, which is possible by Theorem \ref{thmbergexp} (or Corollary \ref{corthmbergexp}). The slope stability of $\mathcal{E}$ (or its reformulation as uniform stability as in Section \ref{ssstaust}), together with some uniform estimates for Fubini--Study metrics established in Section \ref{FSestim}, implies that the sequence $\{ h_i \}_{i \in \mathbb{N}}$ must contain a subsequence that converges in the $C^p$-topology for (any fixed) $p \ge 2$ (Proposition \ref{pmdproperhp}), up to slightly modifying the sequence $\{ h_i \}_{i \in \mathbb{N}}$ as indicated in Lemma \ref{lmdproperhp}, which is shown to converge to a well-defined smooth Hermitian--Einstein metric (Proposition \ref{propregmin}). \\

{\small 
\noindent {\bf Acknowledgments.} {\small The work of both authors was carried out in the framework of the Labex Archim\`ede (ANR-11-LABX-0033) and of the A*MIDEX project (ANR-11-IDEX-0001-02), funded by the ``Investissements d'Avenir" French Government programme managed by the French National Research Agency (ANR). The first named author was supported by the Post-doc position in Complex and Symplectic Geometry ``Bando de Bartolomeis 2017'', at Dipartimento di Matematica e Informatica U.~Dini, Universit\`a di Firenze, dedicated to the memory of Professor Paolo de Bartolomeis and funded by his family, and also by Universit\`a di Firenze, Progetto di ricerca scientifica d'ateneo (ex quota 60\%) Anno 2017: ``Geometria differenziale, algebrica, complessa e aritmetica'', and by JSPS KAKENHI Grant Number 19K14524. The second named author was also partially supported by the ANR project EMARKS, decision No ANR-14-CE25-0010 and a D\'el\'egation CNRS at UMI 3457, Montr\'eal.
}}

\tableofcontents

\section*{Notation}
We largely follow the notation that was used in \cite{HK1}. Throughout, $(X,L)$ stands for a polarised smooth projective variety over $\mathbb{C}$ of complex dimension $n$. We further assume that $L$ is very ample and often write it as $\mathcal{O}_X (1)$, and $L^{\otimes k}$ as $\mathcal{O}_X (k)$. We work with a fixed K\"ahler metric $\omega \in c_1 (L)$ on $X$ defined by a hermitian metric $h_L$ on $L$.

We write $\mathcal{O}_X$ for the sheaf of rings of holomorphic functions on $X$, $C^{\infty}_X$ for the one of $\mathbb{C}$-valued $C^{\infty}$-functions on $X$.

Throughout in this paper, coherent sheaves of $\mathcal{O}_X$-modules will be denoted by calligraphic letters (e.g.~$\mathcal{E}$). Given a coherent sheaf $\mathcal{F}$ of $\mathcal{O}_X$-modules, $H^0(X, \mathcal{F}) = H^0 (\mathcal{F})$ denotes the set of global sections of $\mathcal{F}$; for example, we shall often write $H^0 (\mathcal{F}(k))$ for $H^0(X , \mathcal{F} \otimes L^{\otimes k})$ for any coherent sheaf $\mathcal{F}$ on $X$. When $\mathcal{E}$ is locally free, we write $\Gamma_{C^{\infty}_X}(X, \mathcal{E}) = \Gamma_{C^{\infty}_X} (\mathcal{E})$ for the set of $C^{\infty}$-sections of the complex vector bundle $\mathcal{E}$, and $\mathrm{End}_{C^{\infty}_X} (\mathcal{E})$ for the ones of the vector bundle $\mathcal{E}^{\vee} \otimes \mathcal{E}$.

Unlike in \cite{HK1}, we use the same symbol $\mathcal{E}$ to denote a locally free sheaf of $\mathcal{O}_X$-modules and a complex $C^{\infty}$ vector bundle that underlies it.

We write $N_k$, or simply $N$, for $\dim_{\mathbb{C}} H^0 (\mathcal{E}(k))$.

\section{Preliminaries}
\subsection{Slope stability}

We recall the following notions, where we follow the standard notation to write $\mathrm{End}_{\mathcal{O}_X} (\mathcal{E}) := H^0 ( X, \mathcal{E} nd_{\mathcal{O}_X} (\mathcal{E}) )$ with $\mathcal{E} nd_{\mathcal{O}_X} (\mathcal{E}) := \mathcal{E}^{\vee} \otimes_{\mathcal{O}_X} \mathcal{E}$.

\begin{definition}
A holomorphic vector bundle $\mathcal{E}$ is said to be
\begin{enumerate}
	\item \textbf{reducible} if it can be written as a direct sum of two or more nontrivial holomorphic subbundles as $\mathcal{E} = \bigoplus_j \mathcal{E}_j$;
	\item \textbf{irreducible} if it is not reducible;
	\item \textbf{simple} if $\mathrm{End}_{\mathcal{O}_X}(\mathcal{E}) = \mathbb{C}$.
\end{enumerate}
\end{definition}
It follows immediately that simple vector bundles are irreducible. The following fact is well-known.
\begin{lem} \textup{(cf.~\cite[Corollary 5.7.14]{Kobook})} \label{stablesimple}
	If $\mathcal{E}$ is slope stable, then it is simple.
\end{lem}

Let $(X,L)$ be a polarised smooth projective variety of complex dimension $n$, and $\mathcal{E}$ be a holomorphic vector bundle over $X$ of rank $r$. Throughout in this paper, we assume $r > 1$.

\begin{definition} \label{defrkdeg}
	Let $\mathcal{F}$ be a coherent sheaf on $X$. Its \textbf{slope} $\mu (\mathcal{F}) = \mu_L (\mathcal{F}) \in \mathbb{Q}$ is defined as
	\begin{equation*}
		\mu (\mathcal{F}) := \frac{\deg (\mathcal{F})}{\mathrm{rk} (\mathcal{F})},
	\end{equation*}
	where $\mathrm{rk} (\mathcal{F}) \in \mathbb{N}$ is the rank of $\mathcal{F}$ where it is locally free (cf.~\cite[p.11]{H-L}), and $\deg (\mathcal{F}) \in \mathbb{Z}$ is defined as $\int_X c_1 (\det \mathcal{F}) c_1(L)^{n-1}/(n-1)!$ (where $\det \mathcal{F}$ is a line bundle defined in terms of the locally free resolution of $\mathcal{F}$, see \cite{detdiv}, \cite[V.6]{Kobook}  for details).
\end{definition}

In general we should define the rank and degree in terms of the coefficients of the Hilbert polynomial \cite[Definitions 1.2.2 and 1.2.11]{H-L} and hence they are a priori rational numbers. For us, however, they are defined as integers and the above definition suffices since $X$ is smooth.

The following stability notion was first introduced by Mumford for Riemann surfaces, which was generalised to higher dimensional varieties by Takemoto by choosing a polarisation $L$.

\begin{definition}[Slope stability] \label{defmtstab}
	A holomorphic vector bundle $\mathcal{E}$ is said to be {\bf slope stable} (or Mumford--Takemoto stable) if for any coherent subsheaf $\mathcal{F} \subset \mathcal{E}$ with $0 < \mathrm{rk} (\mathcal{F}) < \rk (\mathcal{E})$ we have $\mu(\mathcal{E}) > \mu(\mathcal{F})$. $\mathcal{E}$ is said to be {\bf slope semistable} if the same condition holds with non-strict inequality, and {\bf slope polystable} if it is a direct sum of slope stable bundles with the same slope.
\end{definition}

Finally, recall the following definition (see \cite[Definitions 1.7.1 and 1.7.3]{H-L}).

\begin{definition}
	A coherent sheaf $\mathcal{F}$ is said to be \textbf{$k$-regular} if $H^i (\mathcal{F}(k-i)) =0$ for all $i >0$. The \textbf{Castelnuovo--Mumford regularity} of $\mathcal{F}$ is the integer defined by
	\begin{equation*}
		\mathrm{reg}(\mathcal{F}) := \inf_{k \in \mathbb{Z}} \{ \mathcal{F} \text{ is $k$-regular.} \} .
	\end{equation*}
\end{definition}

\subsection{Fubini--Study metrics} \label{scfsmetcs}

We recall some basic facts about the Fubini--Study metrics. The reader is referred to \cite{M-M,P-S03,Wang1}, and also \cite[Section 1]{HK1}, for further details of what is presented below.

The key ingredient is the following vector bundle version of the Kodaira embedding. Suppose that $\mathcal{E}$ is $k$-regular and $\rk (\mathcal{E})=r$. Then $\mathcal{E}(k)$ is globally generated, and hence there exists a holomorphic map 
\begin{equation*}
	\Phi: X \to \mathrm{Gr}(r , H^0(\mathcal{E}(k))^{\vee})
\end{equation*}
to the Grassmannian of $r$-planes (rather than quotients), such that the pullback under $\Phi$ of the universal bundle (i.e.~the dual of the tautological bundle) is isomorphic to $\mathcal{E}(k)$. 

Recall that positive definite hermitian forms on $H^0 (\mathcal{E}(k))$ define hermitian metrics on the universal bundle on $\mathrm{Gr}(r , H^0(\mathcal{E}(k))^{\vee})$, called the Fubini--Study metric on the Grassmannian (see e.g.~\cite[Section 5.1.1]{M-M}). By pulling them back by $\Phi$, we have hermitian metrics on $\mathcal{E}$ that are also called Fubini--Study metrics, which can also be defined as follows.

\begin{definition} \label{deffsdual}
	Suppose that we fix a reference hermitian metric $h_{\mathrm{ref}}$ on $\mathcal{E}$ and $h_L$ on $\mathcal{O}_X(1)$. Defining a positive definite hermitian form $H_{\mathrm{ref}}$ on $H^0(\mathcal{E}(k))$ as the $L^2$-inner product with respect to $h_{\mathrm{ref}} \otimes h_L^{\otimes k}$, there exist $C^{\infty}$-maps
	\begin{equation*}
	Q : \mathcal{E} (k) \to H^0 (\mathcal{E}(k)) \otimes C^{\infty}_X,
\end{equation*}
and
\begin{equation*}
	Q^* :  \overline{H^0 (\mathcal{E}(k))^{\vee}} \otimes \overline{C^{\infty}_X} \to \overline{\mathcal{E}(k)^{\vee}},
\end{equation*}
such that the hermitian metric
	\begin{equation*}
		h_k := Q^*Q \in \Gamma_{C^{\infty}_X} (\mathcal{E}^{\vee} \otimes \overline{\mathcal{E}^{\vee}} )
	\end{equation*}
	agrees with the pullback by $\Phi$ of the Fubini--Study metric on the universal bundle over the Grassmannian, defined by the hermitian form $H_{\mathrm{ref}}$ on $H^0(\mathcal{E}(k))$. The metric $h_k = Q^*Q$ is called the (reference) \textbf{Fubini--Study metric} on $\mathcal{E}$ defined by $H_{\mathrm{ref}}$.
	
	Moreover, given $\sigma \in GL(H^0(\mathcal{E}(k))^{\vee})$, the hermitian metric
	\begin{equation*}
		h_{\sigma} := Q^* \sigma^* \sigma Q \in \Gamma_{C^{\infty}_X} (\mathcal{E}^{\vee} \otimes \overline{\mathcal{E}^{\vee}} )
	\end{equation*}
	agrees with the one defined by the hermitian form $\sigma^* \circ H_{\mathrm{ref}} \circ \sigma$ on $H^0(\mathcal{E}(k))$. We shall also write $\sigma^* \sigma$ for $\sigma^* \circ H_{\mathrm{ref}} \circ \sigma$ for notational simplicity. The metric $h_{\sigma}$ is called the \textbf{Fubini--Study metric} on $\mathcal{E}$ defined by the positive hermitian form $\sigma^* \sigma$ on $H^0(\mathcal{E}(k))$.
\end{definition}

For the above to be well-defined, we need to ensure that the maps $Q^*$ and $Q$ with the stated properties do exist, and this is indeed well-known to be true (cf.~\cite[Remark 3.5]{Wang1} or \cite[Theorem 5.1.16]{M-M}).

\begin{rem} \label{qexplicitk}
When it is necessary to make the exponent $k$ more explicit, we also write $Q^*(k)$ for $Q^*$ and $Q(k)$ for $Q$.
\end{rem}

Definition \ref{deffsdual} allows us to associate a hermitian metric $FS(H)$ on $\mathcal{E}$ to a positive definite hermitian form $H$ on $H^0 (\mathcal{E}(k))$; we simply choose $\sigma \in GL (H^0(\mathcal{E} (k))^{\vee})$ so that $H = \sigma^* \circ H_{\mathrm{ref}} \circ \sigma = \sigma^* \sigma$, and define $FS(H) = h_{\sigma}$ as above. Recall that there is also an alternative definition (\cite{Wang1}, see also \cite[Theorem 5.1.16]{M-M}) of $FS(H)$ by means of the equation
\begin{equation} \label{deffseq}
	\sum_{i=1}^N s_i\otimes s_i^{*_{FS(H)}} = \Id_{\mathcal{E}}
\end{equation}
where $\{s_i\}$ is an $H$-orthonormal basis for $H^0 (\mathcal{E}(k))$ and $s_i^{*_{FS(H)}}$ is the $FS(H)$-metric dual of $s_i$, and $N = N_k = \dim_{\mathbb{C}}H^0(\mathcal{E}(k))$.

Suppose that we write $\mathcal{H}_k$ for the subset of $\mathcal{H}_{\infty}$ consisting of all Fubini--Study metrics defined by hermitian forms on $H^0 (\mathcal{E}(k))$. Although $\mathcal{H}_k$ is a very small subset of $\mathcal{H}_{\infty}$, it is well-known that any element in $\mathcal{H}_{\infty}$ can be approximated by the elements in $\mathcal{H}_k$ by choosing $k$ to be sufficiently large. To prove this result, we need to invoke the well-known asymptotic expansion of the Bergman kernel (also called the Tian--Yau--Zelditch expansion), as stated below.

\begin{thm}[Asymptotic expansion of the Bergman kernel]\label{thmbergexp}
	Suppose that we fix a positively curved hermitian metric $h_L$ on $\mathcal{O}_X (1)$ and take $h \in \mathcal{H}_{\infty}$, which defines a positive definite $L^2$-hermitian form $H$ on $H^0(\mathcal{E}(k))$. The \textbf{Bergman kernel} $B_{k}(h) \in \mathrm{End}_{C^{\infty}_X}(\mathcal{E})$, defined by the equation $B_k(h) \circ FS(H) = h$, satisfies the asymptotic expansion $B_k(h) = \Id_{\mathcal{E}} + O(k^{-1})$; more precisely, there exists a constant $C(h, h_L, p)>0$ depending on $h$, $h_L$, and $p \in \mathbb{N}$ such that
	\begin{equation*}
		\Vert B_k(h) - \Id_{\mathcal{E}} \Vert_{C^p} \le C(h,h_L,p)/k,
	\end{equation*}
	where $\Vert \cdot \Vert_{C^p}$ is the $C^p$-norm on $\mathrm{End}_{C^{\infty}_X}(\mathcal{E})$.
\end{thm}

\begin{cor} \label{corthmbergexp}
	For any $h \in \mathcal{H}_{\infty}$ and $p \in \mathbb{N}$ there exists a sequence $\{ h_{k,p} \}_{k \in \mathbb{N}}$ of Fubini--Study metrics with $ h_{k,p} \in \mathcal{H}_k$ such that $h_{k,p} \to h$ as $k \to + \infty$ in the $C^p$-norm, i.e.
	\begin{equation*}
		\mathcal{H}_{\infty} = \overline{\cup_{k>>0} \mathcal{H}_k}.
	\end{equation*}
\end{cor}

Although the above corollary is all we need in this paper, there is no known proof of it that is not based on Theorem \ref{thmbergexp}, which is a deep result in analysis. Corollary \ref{corthmbergexp}, which we rely on as a foundational result, plays an important role in what follows.

For the proof of Theorem \ref{thmbergexp}, the reader is referred to \cite{Catlin}, \cite{Wang2}; see also the book \cite{M-M} and references therein. An elementary proof can be found in \cite{BBS2008}.

\subsection{Donaldson's functional and the Hermitian--Einstein equation}\label{defMDon}
Let $(X,L)$ be a polarised smooth complex projective variety, as before, with $\Vol_L := \int_X c_1(L)^n/n!$, and $\mathcal{E}$ be a holomorphic vector bundle of rank $r >1$. The following functional plays a central role in this paper and the prequel \cite{HK1}.

\begin{definition}
 Given two hermitian metrics $h_0$ and $h_1$ on $\mathcal{E}$, \textbf{the Donaldson functional} ${\mathcal M}^{Don} : \mathcal{H}_{\infty}\times \mathcal{H}_{\infty} \to \mathbb{R}$ is defined as 
 \begin{equation*}
 	{\mathcal M}^{Don}(h_1,h_0) := \int_0^1 dt \int_X \tr \left( h_t^{-1} \partial_t h_t \cdot \left(\Lambda_\omega F_t -  \frac{\mu(\mathcal{E})}{\mathrm{Vol}_L} \Id_{\mathcal{E}} \right) \right) \frac{\omega^{n}}{n!},
 \end{equation*}
 where $\{ h_t \}_{0 \le t \le 1} \subset \mathcal{H}_{\infty}$ is a smooth path of hermitian metrics between $h_0$ to $h_1$, and $F_t$ denotes ($\ai / 2 \pi$) times the Chern curvature of $h_t$ with respect to the fixed holomorphic structure of $\mathcal{E}$. Our convention is that the second argument of ${\mathcal M}^{Don}$ is fixed as a reference metric.
\end{definition}

We recall some basic properties of this functional that are established in \cite{Don1985}, while the reader is also referred to \cite[Section 6.3]{Kobook} for more details. First of all it is well-defined, i.e.~does not depend on the path $\{ h_t \}_{0 \le t \le 1}$ chosen to connect $h_0$ and $h_1$ (cf.~\cite[Lemma 6.3.6]{Kobook}); note that this easily implies the following \textbf{cocycle property}
\begin{equation} \label{cocyclemdon}
	\mathcal{M}^{Don} (h_2,h_0) = \mathcal{M}^{Don} (h_2,h_1)+ \mathcal{M}^{Don} (h_1,h_0) ,
\end{equation}
for any $h_0 , h_1 , h_2 \in \mathcal{H}_{\infty}$. In particular, this implies that $\mathcal{M}^{Don}(-,h_0)$ is invariant under an overall constant scaling, since
\begin{equation} \label{mdsclinv}
		{\mathcal M}^{Don}(e^{c} h,h_0) = {\mathcal M}^{Don}(e^{c}h,h)+ {\mathcal M}^{Don}(h,h_0) = {\mathcal M}^{Don}(h,h_0),
\end{equation}
by recalling $\mathcal{M}^{Don} (e^{c} h,h)=0$ for any constant $c \in \mathbb{R}$ \cite[Lemma 6.3.23]{Kobook}.

Second, the critical point of $\mathcal{M}^{Don}(-,h_{\mathrm{ref}})$ is the following object.

\begin{definition}
	A hermitian metric $h \in \mathcal{H}_{\infty}$ is called a \textbf{Hermitian--Einstein metric} if it satisfies $$\Lambda_\omega F_h = \frac{\mu(\mathcal{E})}{\Vol_L} \Id_{\mathcal{E}},$$
where $\Lambda_\omega$ is the contraction with respect to the K\"ahler metric $\omega$ on $X$.
\end{definition}

This is a consequence of the following lemma.

\begin{lem} \textup{(cf.~\cite[Section 1.2]{Don1985})} \label{derivmdon}
	Fixing a reference metric $h_{\mathrm{ref}} \in \mathcal{H}_{\infty}$, we have
	\begin{equation*}
		\frac{d}{dt} \mathcal{M}^{Don} (h_t , h_{\mathrm{ref}}) = \int_X \tr \left( h_t^{-1} \partial_t h_t \left( \Lambda_{\omega} F_t - \frac{\mu(\mathcal{E})}{\mathrm{Vol}_L} \Id_{\mathcal{E}} \right) \right) \frac{\omega^{n}}{n!}
	\end{equation*}
	for a path $\{ h_t \}_{0 \le t \le 1} \subset \mathcal{H}_{\infty}$ of smooth hermitian metrics with $h_0 =h_{\mathrm{ref}}$.
\end{lem}

The final important point is that $\mathcal{M}^{Don}$ is \textit{convex} along geodesics in $\mathcal{H}_{\infty}$, where the geodesics are defined as follows (cf.~\cite[Section 6.2]{Kobook}).

\begin{definition} \label{defgeod}
A path $\{ h_s \}_{s \in \mathbb{R}} \subset \mathcal{H}_{\infty}$ is called a \textbf{geodesic} in $\mathcal{H}_{\infty}$ if it satisfies
\begin{equation} \label{geodcalh}
	\partial_s (h_s^{-1} \partial_s h_s ) =0,
\end{equation}
as an equation in $\mathrm{End}_{C^{\infty}_X}(\mathcal{E})$; an overall constant scaling $h_s := e^{bs}h_0$ for some $b \in \mathbb{R}$ will be called a \textbf{trivial} geodesic.	
\end{definition}
An important point is that $\mathcal{H}_{\infty}$ is \textit{geodesically complete}; for any $h_0, h_1 \in \mathcal{H}_{\infty}$ there exists a geodesic path $\{ h_s \}_{0 \le s \le 1}$ connecting them; this can be proved by writing the geodesic explicitly as $h_s = \exp( s \log h_1 h^{-1}_0) h_0$. Thus, geodesic convexity of $\mathcal{M}^{Don}$ and geodesic completeness of $\mathcal{H}_{\infty}$ together imply that a critical point of $\mathcal{M}^{Don}$ has to attain the global minimum. The precise statement is as follows.

\begin{prop} \textup{(cf.~\cite[Section 6.3]{Kobook})} \label{lemmdonconvsm}
 The functional $\mathcal{M}^{Don}$ is convex along geodesics in $\mathcal{H}_{\infty}$, and its critical point (if exists) attains the global minimum. Moreover, $\mathcal{M}^{Don}$ is strictly convex along nontrivial geodesics if $\mathcal{E}$ is irreducible (in particular if $\mathcal{E}$ is simple) and in this case the critical point (if exists) is unique up to an overall constant scaling.
\end{prop}

Although this is a well-known result, we provide a self-contained proof of it in the appendix (see Proposition \ref{lemmdonconvH}) so as to make clear that its proof carries over to give a slightly stronger version that is adapted to the $C^p$-completion of $\mathcal{H}_{\infty}$; see Proposition \ref{lemmdonconv} in the appendix for the precise statement.

\section{Elementary uniform estimates for Fubini--Study metrics}\label{FSestim}

We prove some elementary estimates that we need in the proof of Theorem \ref{mainthmhk}. The content of the main result in this section (Proposition \ref{parczestzeta}) can be summarised as follows: if $Q^*(k)Q(k)$ (as introduced in Section \ref{scfsmetcs}, see also Remark \ref{qexplicitk}) converges in $C^{2p}$, the $C^p$-norms of Fubini--Study metrics $Q^*(k) e^{\zeta^*} e^{\zeta} Q(k)$ can be controlled by the operator norm of $\zeta \in \mathfrak{gl} (H^0(\mathcal{E}(k))^{\vee})$ \textit{irrespectively of $k$}. This follows form a uniform estimate stated in Proposition \ref{parczest}, and the results proved in this section will play an important role in Section \ref{stabtocan}. \\

We set our notational convention in this section as follows. We fix a basis $(s_1 , \dots s_N)$ for $H^0(\mathcal{E}(k))$ for each $k$, an orthonormal frame $(e_1 , \dots , e_r)$ with respect to $h_{\mathrm{ref}}$, and also trivialising open sets in $X$ that cover $X$. This means that $Q(k)^*$ (resp.~$Q(k)$) will be regarded as an $r \times N$ (resp.~$N \times r$) matrix that depends smoothly on the base coordinates, on each trivialising neighbourhood. Indices $i,j, \dots$ will run from 1 to $N$ (which grows as $k$ grows), and indices $\alpha, \beta, \dots$ will run from 1 to $r$ (rank of $\mathcal{E}$). Moreover, we write $\partial_x$ for the shorthand to denote differentiation in base coordinates; $\partial_z$ and $\bar{\partial}_z$ will also be used for the derivatives in (local) holomorphic and anti-holomorphic coordinates on $X$. Our computation below will all be local, with the conventions set as above.

\begin{prop} \label{parczest}
	Suppose that we have a sequence $\{ h_k \}_{k \in \mathbb{N}}$ of Fubini--Study metrics $h_k := Q(k)^*Q (k) \in \mathcal{H}_k$ that converges to $h_{\mathrm{ref}}$ in the $C^{2p}$-norm as $k \to \infty$. Then there exists a constant $C_l = C(h_{\mathrm{ref}}, l)>0$ which depends only on $h_{\mathrm{ref}}$ and $0 \le l \le p$, $l \in \mathbb{Z}$, such that
	\begin{equation*}
		\sum_{i=1}^N \sum_{\alpha =1}^r | \partial_x^l  Q(k)^*_{\alpha i}|^2 < C_l
	\end{equation*}
	holds uniformly for all $k$, where the notational convention above is understood.
\end{prop}

\begin{proof}
Writing $(h_k)_{\alpha \beta}$ for the $(\alpha , \beta)$-th entry of $h_k$, we first observe $(h_k)_{\alpha \beta} = \sum_{i=1}^N Q (k)^*_{\alpha i} Q (k)_{ i \beta} = \sum_{i=1}^N Q (k)^*_{\alpha i } \overline{Q (k)^*_{\beta i} }$. Thus we have
	\begin{equation} \label{eqcsleqzr}
		\sum_{i=1}^N | Q (k)^*_{\alpha i}|^2 = (h_k)_{\alpha \alpha} \le 2 (h_{\mathrm{ref}})_{\alpha \alpha},
	\end{equation}
	by $h_k \to h_{\mathrm{ref}}$ ($k \to \infty$) in $C^{2p}$, which establishes the case $l=0$ of our claim by summing over $\alpha = 1 , \dots , r$.
	
	We now proceed to the case $l = 1$. We shall prove
	\begin{align}
		\sum_{i=1}^N \sum_{ \alpha =1}^r | \partial_z Q(k)^*_{i \alpha} |^2  &\le C_1 \label{clmdelbar} \\
		\sum_{i=1}^N \sum_{ \alpha =1}^r | \bar{\partial}_z Q(k)^*_{i \alpha} |^2  &\le C_1 \label{clmdel},
	\end{align}
	for a constant $C_1$ that does not depend on $k$.
	
	Suppose that we write $(e(k)_1 , \dots , e(k)_r)$ for an $h_k$-orthonormal frame, regarded as a vector-valued function on (each trivialising neighbourhood of) $X$.

	Observe that the local frame $(e_1 , \dots , e_r)$ (with respect to $h_{\mathrm{ref}}$) and the one $(e(k)_1 , \dots , e(k)_r)$ (with respect to $h_k$) can be related as
	\begin{equation*}
		\begin{pmatrix}
			e(k)_1 \\
			\vdots \\
			e(k)_r
		\end{pmatrix}
		= P(k)
		\begin{pmatrix}
			e_1 \\
			\vdots \\
			e_r
		\end{pmatrix}
	\end{equation*}
	by an $r \times r$-matrix valued function $P(k)$ defined on each trivialising neighbourhood, which we may assume is hermitian. Since $h_k \to h_{\mathrm{ref}}$ in $C^{2p}$ (and hence in $C^p$) as $k \to \infty$, we can choose the above orthonormal frames so that
	\begin{equation} \label{convpkkrdlt}
		P(k)_{\alpha \beta} - \delta_{\alpha \beta} \to 0
	\end{equation}
	in $C^p$ as $k \to \infty$, where $\delta_{\alpha \beta}$ is the Kronecker delta. In particular, $\Vert P(k)_{\alpha \beta} - \delta_{\alpha \beta}\Vert_{C^p}$ is bounded uniformly for all $k$.
	
	We write $\tilde{Q}(k)^*_{\alpha i}$ for the $r \times N$-matrix valued function, representing $Q(k)^*$ with respect to the $h_k$-orthonormal frame $(e(k)_1 , \dots , e(k)_r)$ and the fixed basis $(s_1 , \dots s_N)$ for $H^0(\mathcal{E}(k))$. By using the matrix $P(k)$ above, this can be written explicitly as
	\begin{equation*}
		Q(k)^*_{\alpha i} = \sum_{\beta =1}^r P(k)_{\alpha \beta} \tilde{Q}(k)^*_{\beta i},
	\end{equation*}
	and
	\begin{equation} \label{relqqtildep}
		Q(k)_{i \alpha} = \sum_{\beta=1}^r  \tilde{Q}(k)_{i \beta} P(k)^*_{\beta \alpha}.
	\end{equation}

	By the definitions of $Q(k)$ and $P(k)$ (recall also \cite[Remark 3.5]{Wang1}, \cite[Theorem 5.1.16]{M-M} and \cite[Lemma 1.13]{HK1}), we have
	\begin{equation} \label{relqeksi}
		\sum_{\alpha=1}^r \tilde{Q}(k)_{i \alpha} e_{\alpha} (k) = s_i (x).
	\end{equation}
	Since $s_i$ is a holomorphic section, we apply $\bar{\partial}_z$ on both sides of the above equation to get
	\begin{equation} \label{equnifahol}
		\sum_{\alpha=1}^r ( \bar{\partial}_z \tilde{Q}(k)_{i \alpha}) e_{\alpha} (k) + \sum_{\alpha=1}^r \tilde{Q}(k)_{i \alpha} ( \bar{\partial}_z e_{\alpha} (k) ) = 0 .
	\end{equation}
	On the other hand, observe that there exists a (local) $r \times r$-matrix $R(k,1)$ such that
	\begin{equation} \label{eqdefrkone}
		\bar{\partial}_z e_{\alpha} (k) = \sum_{\beta=1}^r e_{\beta} (k) R_{\beta \alpha} (k,1) ,
	\end{equation}
so that we can re-write (\ref{equnifahol}) as
	\begin{equation*}
		\sum_{\alpha=1}^r ( \bar{\partial}_z \tilde{Q}(k)_{i \alpha}) e_{\alpha} (k)  = - \sum_{\alpha, \beta =1}^r \tilde{Q}(k)_{i \alpha}  R_{\beta \alpha} (k,1) e_{\beta} (k) .
	\end{equation*}
	The convergence (\ref{convpkkrdlt}) in $C^p$ means that $\bar{\partial}_z e_{\alpha}(k)$ can be controlled uniformly for all $k$, and hence the operator norm of $R(k,1)$ can be controlled uniformly for all $k$. 
	We now apply a map $h_k ( - , e_{\gamma}): \mathcal{E} \to \mathbb{C}$ to the above equation. Since $(e(k)_1 , \dots , e(k)_r)$ is an $h_k$-orthonormal frame, we get
	\begin{equation*}
		\bar{\partial}_z \tilde{Q}(k)_{i \gamma} = - \sum_{\alpha=1}^r \tilde{Q}(k)_{i \alpha}  R_{\gamma \alpha} (k,1) ,
	\end{equation*}
	and by taking the complex conjugate, we get
\begin{equation*}
		\partial_z \tilde{Q}(k)^*_{\gamma i} = - \sum_{\alpha=1}^r \tilde{Q}(k)^*_{\alpha i}  \overline{R_{\gamma \alpha} (k,1) }.
	\end{equation*}
	These two equalities imply
	{\small \begin{align*}
		\sum_{i=1}^N | \partial_z \tilde{Q}(k)^*_{\alpha i } |^2 &=\sum_{i=1}^N (\partial_z \tilde{Q}(k)^*_{\alpha i}) (\bar{\partial}_z \tilde{Q}(k)_{i \alpha}) \\
		&= \sum_{i=1}^N \left( \sum_{\beta=1}^r \tilde{Q}(k)^*_{\beta i} \overline{R_{\alpha \beta} (k,1)}  \right)  \left( \sum_{\beta =1}^r \tilde{Q}(k)_{i \beta}  R_{\alpha \beta} (k,1)  \right) \\
		&= \sum_{i=1}^N \left| \sum_{\beta=1}^r \tilde{Q}(k)^*_{\beta i} \overline{R_{\alpha \beta} (k,1)}   \right|^2 \\
		&\le C(R(k,1)) \left( \sum_{i=1}^N \sum_{\alpha=1}^r |\tilde{Q}(k)^*_{i \alpha}|^2 \right).
	\end{align*}}
	In the above, the constant $C(R(k,1))$ depends only on the operator norm of $R(k,1)$, which can be bounded uniformly for all $k$ by (\ref{convpkkrdlt}). Thus, we have proved that there exists a constant $\tilde{C}_1$ that does not depend on $k$ such that
	\begin{equation*}
		\sum_{i=1}^N | \partial_z \tilde{Q}(k)^*_{i \alpha} |^2  \le \tilde{C}_1 C_0^2,
	\end{equation*}
	by recalling the estimate for the case $l=0$. Recalling (\ref{convpkkrdlt}), (\ref{relqqtildep}), and the estimate for $l=0$, we get $$\sum_{i=1}^N \sum_{ \alpha =1}^r | \partial_z Q(k)^*_{i \alpha} |^2  \le C_1,$$ for some $C_1 > 0$ that does not depend on $k$, which is what we claimed in (\ref{clmdelbar}).

	\medskip
	
	We now prove (\ref{clmdel}). Similarly to (\ref{equnifahol}) we have
	\begin{equation} \label{eqqthols}
		\sum_{\alpha=1}^r ( {\partial}_z \tilde{Q}(k)_{i \alpha}) e_{\alpha} (k) + \sum_{\alpha=1}^r \tilde{Q}(k)_{i \alpha} ( {\partial}_z e_{\alpha} (k) ) = \partial_z s_i (x),
	\end{equation}
	and hence
	\begin{equation} \label{eqthlsdq}
		\partial_z \tilde{Q}(k)_{i \gamma} = - \sum_{\alpha=1}^r \tilde{Q}(k)_{i \alpha}  R_{\gamma \alpha} (k,1) + (\partial_z s_i)_{\gamma},
	\end{equation}
	where we wrote $(\partial_z s_i)_{\gamma} := h_k(\partial_z s_i (x) ,e_{\gamma})$. Again by taking the complex conjugate, we have
	\begin{equation} \label{eqthlsdbq}
		\bar{\partial}_z \tilde{Q}(k)^*_{\gamma i} = - \sum_{\alpha=1}^r \tilde{Q}(k)^*_{\alpha i}  \overline{R_{\gamma \alpha} (k,1)} + (\overline{\partial_z s_i})_{\gamma}.
	\end{equation}

	Thus, in addition to the previous argument for (\ref{clmdelbar}), it is necessary to bound $\sum_{\alpha , i} | (\partial_z s_i )_{\alpha}|^2$ uniformly for all $k$ (see \cite[Section 5.4]{yhextremal} for a similar argument). In fact, we bound all the $l$-th derivatives $\sum_{\alpha , i} | (\partial_z^l s_i )_{\alpha}|^2$, for $1 \le l \le p$, uniformly for all $k$ by assuming the convergence $h_k \to h_{\mathrm{ref}}$ in $C^{2p}$. We fix a local trivialisation and write $h_k = e^{K}$ for a smooth local $r \times r$ hermitian form $K \in \mathrm{End}_{C^{\infty}_X} (\mathcal{E})$. Then, we can re-write the equation (\ref{deffseq}) as
	\begin{equation*}
		\sum_{i=1}^N (s_i)_{\alpha} \otimes (e^K)_{\beta \gamma} (\bar{s}_i)_{\gamma} = \delta_{\alpha \beta},
	\end{equation*}
	where $\alpha , \beta , \gamma$ denotes indices running from 1 to $r$ to denote the endomorphism componentwise, and $s_i$ is regarded as a vector-valued holomorphic function. Applying $e^{-K}$, this implies
	\begin{equation*}
		\sum_{i=1}^N (s_i)_{\alpha} \otimes  (\bar{s}_i)_{\beta} = (e^{-K})_{\alpha \beta}.
	\end{equation*}
	We apply $\partial_z \bar{\partial}_z$ to this equation to get
	\begin{equation*}
		\sum_{i=1}^N (\partial_z s_i)_{\alpha} \otimes  (\bar{\partial}_z \bar{s}_i)_{\beta} = (\partial_z \bar{\partial}_z e^{-K})_{\alpha \beta},
	\end{equation*}
	by noting $\bar{\partial}_z s_i =0$ and its complex conjugate. We can iterate this procedure so that
	\begin{equation*}
		\sum_{i=1}^N (\partial_z^l s_i)_{\alpha} \otimes  (\bar{\partial}_z^l \bar{s}_i)_{\beta} =  (\partial_z^l \bar{\partial}_z^l e^{-K})_{\alpha \beta}.
	\end{equation*}
	We take the $h_k$-metric trace of this to get
	\begin{equation} \label{bdhgders}
		\sum_{\alpha =1}^r \sum_{i=1}^N |(\partial_z^l s_i)_{\alpha}|^2 \le C(2l, h_k),
	\end{equation}
	where $C(2l, h_k)$ is a constant which depends only on the $C^{2l}$-norm of the hermitian metric $h_k$, which is under control as long as $l \le p$ by assuming the convergence $h_k \to h_{\mathrm{ref}}$ in $C^{2p}$.

Thus, (\ref{eqthlsdq}), (\ref{eqthlsdbq}), (\ref{bdhgders}), and Cauchy--Schwarz imply
{\small
	\begin{align*}
		&\sum_{i=1}^N | \bar{\partial}_z \tilde{Q}(k)^*_{i \alpha} |^2 \\
		&= \sum_{i=1}^N \left( \sum_{\beta=1}^r \tilde{Q}(k)^*_{i \beta}  \overline{R_{\alpha \beta} (k,1)} - (\overline{\partial_z s_i})_{\alpha} \right)  \left( \sum_{\beta =1}^r \tilde{Q}(k)_{i \beta}  R_{\alpha \beta} (k,1) - (\partial_z s_i)_{\alpha} \right), \\
		&\le \sum_{i=1}^N \left( \left| \sum_{\beta=1}^r \tilde{Q}(k)^*_{\beta i} \overline{R_{\alpha \beta} (k,1)}   \right| + |(\partial_z s_i)_{\alpha}| \right)^2 \\
		&\le  C(R(k,1)) \sum_{i=1}^N \sum_{ \alpha =1}^r | Q(k)^*_{i \alpha} |^2 + 2 \left( C(2,h_k) C(R(k,1)) \sum_{i=1}^N \sum_{ \alpha =1}^r | Q(k)^*_{i \alpha} |^2 \right)^{1/2}\\
		&\quad + C(2,h_k).
	\end{align*}}
Proceeding exactly as we did to establish (\ref{clmdelbar}), we thus get $$\sum_{i=1}^N \sum_{ \alpha =1}^r | \bar{{\partial}}_z Q(k)^*_{i \alpha} |^2  \le C_1,$$ for some $C_1 > 0$ uniformly for all $k$, establishing (\ref{clmdel}). This completes the proof of the proposition for the case $l=1$.
	
	For $2 \le l \le p$, we proceed by induction. Differentiating (\ref{relqeksi}) $l$ times, we get
	\begin{equation*}
		\sum_{\alpha=1}^r ( \partial^l_x \tilde{Q}(k)_{i \alpha}) e_{\alpha} (k)  + (\text{terms involving } \partial^m_x \tilde{Q}(k)_{i \alpha} \text{ with } m < l )= \partial^l_x s_i (x) .
	\end{equation*}
	We then proceed exactly as we did for the case $l=1$, by decomposing $\partial_x$ in $\partial_z$ and $\bar{\partial}_z$. We replace (\ref{eqdefrkone}) by $\partial^l_x e_{\alpha} (k) = \sum_{\beta=1}^r e_{\beta} (k) R_{\beta \alpha} (k,l)$ for some $R(k,l)$ which can be controlled uniformly for all $k$ due to (\ref{convpkkrdlt}). Using Cauchy--Schwarz and by induction on $l$, we get the claim for all $2 \le l \le p$ by recalling (\ref{bdhgders}).
\end{proof}

\begin{prop} \label{parczestzeta}
	Suppose that we are given a sequence $\{ h_k \}_{k \in \mathbb{N}}$ of Fubini--Study metrics $h_k := Q(k)^*Q (k) \in \mathcal{H}_k$ converging to $h_{\mathrm{ref}}$ in $C^{2p}$. Then, the Fubini--Study metric $h_{k,\zeta} = Q(k)^* e^{\zeta^*} e^{\zeta} Q(k)$ defined by a hermitian $\zeta \in \mathfrak{gl} (H^0( \mathcal{E}(k))^{\vee})$ satisfies the following bound 
	\begin{equation*}
		\Vert h_{k,\zeta} - h_k \Vert_{C^l} < C_l(\Vert\zeta \Vert_{op} , h_{\mathrm{ref}}) ,
	\end{equation*}
	where $C_l(\Vert \zeta \Vert_{op} , h_{\mathrm{ref}})$ is a constant which depends on the operator norm of $\zeta$, $h_{\mathrm{ref}}$, and $0 \le l \le p$, $l \in \mathbb{Z}$, but not on $k$.
\end{prop}

\begin{proof}
Recalling $\sum_i | Q (k)_{i \alpha}|^2 = (h_k)_{\alpha \alpha}$ as in the proof of Proposition \ref{parczest}, with the notational conventions used therein, we get (up to replacing the basis $(s_1 , \dots , s_N)$ by a unitarily equivalent one if necessary)
	\begin{equation*}
		(h_{k, \zeta})_{\alpha \alpha} = \sum_{i=1}^N e^{2 w_i} | Q (k)_{i \alpha}|^2,
	\end{equation*}
	where $(w_1 , \dots , w_N )$ are the eigenvalues of $\zeta$. This immediately implies
	\begin{equation*}
		 e^{-2 ||\zeta||_{op}} \sum_{i=1}^N | Q (k)_{i \alpha}|^2 < \sum_{i=1}^N e^{2 w_i} | Q (k)_{i \alpha}|^2 < e^{2 ||\zeta||_{op}} \sum_{i=1}^N | Q (k)_{i \alpha}|^2 .
	\end{equation*}
	Now note
	\begin{equation*}
		(h_k)_{\alpha \beta} - (h_{k , \zeta})_{\alpha \beta} = \sum_{i=1}^N \overline{Q (k)_{i \alpha}} (1- e^{2 w_i} ) Q (k)_{ i \beta} .
	\end{equation*}
	Combining the above two estimates with Cauchy--Schwarz (which means that it suffices to evaluate the case $\alpha = \beta$), and also recalling the case $l=0$ of Proposition \ref{parczest},we get
	\begin{equation*}
		|| (h_k)_{\alpha \beta} - (h_{k , \zeta})_{\alpha \beta}||_{C^0} < C_0(|| \zeta||_{op} , h_{\mathrm{ref}}).
	\end{equation*}
	Writing 
	\begin{equation*}
	(h^l_k)_{\alpha \beta} := \sum_{i=1}^N (\partial_x^l Q (k)^*_{\alpha i}) (\partial_x^l Q (k)_{ i \beta}),
	\end{equation*}
	and
	\begin{equation*}
	(h^l_{k, \zeta})_{\alpha \beta} := \sum_{i=1}^N e^{2 w_i} (\partial_x^l Q (k)^*_{\alpha i}) (\partial_x^l Q (k)_{ i \beta}),
	\end{equation*}
	exactly the same argument as above implies
	\begin{equation*}
		|| (h^l_k)_{\alpha \beta} - (h^l_{k , \zeta})_{\alpha \beta}||_{C^0} < C_l(|| \zeta||_{op} , h_{\mathrm{ref}})
	\end{equation*}
	by Proposition \ref{parczest}. By observing that the $C^l$-norm of $(h_k)_{\alpha \beta} - (h_{k , \zeta})_{\alpha \beta}$ can be bounded by a linear combination of $|| (h^m_k)_{\alpha \beta} - (h^m_{k , \zeta})_{\alpha \beta}||_{C^0}$, $m \le l$ (with uniformly bounded coefficients), by means of Cauchy--Schwarz, we get the required estimates.
\end{proof}

\begin{rem} \label{remlbrqlimfs}
A more delicate problem is to establish a \textit{lower} bound for $| Q (k)_{i \alpha}|^2$ that holds independently of $k$, while the uniform upper bound can be obtained as above. With such a lower bound, we can prove inequalities in \cite[Remark 3.10]{HK1}, which in turn proves Hypothesis \ref{conjpczero} that is discussed later in Section \ref{stabtocan}.
\end{rem}

\section{Review of the main results of \cite{HK1}} \label{scrmrhk1}

\subsection{Quot-scheme limit of Fubini--Study metrics} \label{revqslimfs}
We recall some key concepts from \cite{HK1} that we need in what follows, and refer the reader to \cite{HK1} for more details on this section. 

Recall from Section \ref{scfsmetcs} that the Fubini--Study metrics can be written as $h_{\sigma} := Q^* \sigma^* \sigma Q \in \Gamma_{C^{\infty}_X} (\mathcal{E}^{\vee} \otimes \overline{\mathcal{E}^{\vee}} )$ for some $\sigma \in SL(H^0(\mathcal{E}(k))^{\vee})$, up to an overall constant scaling which we ignore for the moment. In particular, choosing $\zeta \in \mathfrak{sl} (H^0(\mathcal{E}(k))^{\vee})$, we have a 1-parameter subgroup (1-PS) $\{ h_{\sigma_t} \}_{t \ge 0}$ of Fubini--Study metrics defined by
\begin{equation*}
	h_{\sigma_t} := Q^* \sigma_t^* \sigma_t Q \in \Gamma_{C^{\infty}_X} (\mathcal{E}^{\vee} \otimes \overline{\mathcal{E}^{\vee}} )
\end{equation*}
with $\sigma_t := e^{\zeta t}$. We call the above $\{ h_{\sigma_t} \}_{t \ge 0}$ the \textbf{Bergman 1-PS} generated by $\zeta$; further, when $\zeta$ has rational eigenvalues, it is called the \textbf{rational Bergman 1-PS}. The main theme of \cite{HK1} is to evaluate the limit of $h_{\sigma_t}$ as $t \to +\infty$ for $\zeta \in \mathfrak{sl} (H^0 (\mathcal{E}(k))^{\vee})$ with rational eigenvalues, in terms of the Quot-scheme limit. Throughout in what follows, we shall assume that the operator norm (i.e. the modulus of the maximum eigenvalue) of $\zeta$ is at most 1, as pointed out in \cite[Remark 2.1]{HK1}.

Suppose that $\zeta \in \mathfrak{sl} (H^0 (\mathcal{E}(k))^{\vee})$ has eigenvalues $w_1 , \dots , w_{\nu} \in \mathbb{Q}$, with the ordering
\begin{equation} \label{ordlambda}
	w_1 > \cdots > w_{\nu} .
\end{equation}
We consider the action of $\zeta$ on $H^0(\mathcal{E}(k))$ which is not the natural dual action, but the one that is natural with respect to certain metric duals (see \cite[(2.6)]{HK1} and the discussion that follows). This yields the weight decomposition
\begin{equation*}
 	H^0 ( \mathcal{E}(k)) = \bigoplus_{i = 1}^{\nu} V_{-w_i , k}
 \end{equation*}
where $\zeta$ acts on $V_{-w_i , k}$ via the $\mathbb{C}^*$-action $T: \mathbb{C}^* \curvearrowright V_{-w_i , k}$ defined by $T \mapsto T^{-w_i}$ (cf.~\cite[Section 2.1]{HK1}). We thus get the filtration
\begin{equation}\label{VT'}
V_{ \le -w_i , k} := \bigoplus_{j= 1}^{i} V_{-w_j, k},
\end{equation}
of $H^0 ( \mathcal{E}(k)) $ by its vector subspaces.

The filtration \eqref{VT'} also gives rise to the one
\begin{equation} \label{filtevs}
	0 \neq \mathcal{E}_{\le -w_{1}} \subset \cdots \subset \mathcal{E}_{\le -w_{\nu}} = \mathcal{E}
\end{equation}
of $\mathcal{E}$ by subsheaves, where $\mathcal{E}_{\le - w_i}$ is a coherent subsheaf of $\mathcal{E}$ defined by the quotient map
\begin{equation*}
	\rho_{\le - w_i} : V_{\le - w_i ,k} \otimes \mathcal{O}_X(-k) \to \mathcal{E}_{\le - w_i}
\end{equation*}
induced from $\rho$. As in \cite[Lemma 2.5]{HK1}, we can modify this filtration on a Zariski closed subset of $X$, to get a filtration
\begin{equation} \label{satfiltevs}
	0 \neq \mathcal{E}'_{\le -w_{1}} \subset \cdots \subset \mathcal{E}'_{\le - w_{\nu}} = \mathcal{E}
\end{equation}
of $\mathcal{E}$ by saturated subsheaves, which will be important later.

Following \cite[Definition 2.2]{HK1}, we can pick a certain subset
\begin{equation} \label{defsubswgt}
\{ w_{\alpha} \}_{\alpha =\hat{1}}^{\hat{\nu}} \subset \{ w_i \}_{i=1}^{\nu}	
\end{equation}
with $\{ \hat{1} , \dots , \hat{\nu} \} \subset \{ 1 , \dots , \nu \}$, by means of the Quot-scheme limit as follows. Considering the quotient map
\begin{equation*}
	\rho : H^0(\mathcal{E}(k)) \otimes \mathcal{O}_X (-k) \to \mathcal{E}
\end{equation*}
for $\mathcal{E}$ and its $\mathbb{C}^*$-orbit
\begin{equation*}
	\rho_T := \rho \circ T^{\zeta} : H^0(\mathcal{E}(k)) \otimes \mathcal{O}_X (-k) \to \mathcal{E}
\end{equation*}
and taking the limit of $T \to 0$, we get a coherent sheaf
\begin{equation*}
	\hat{\rho} : H^0(\mathcal{E}(k)) \otimes \mathcal{O}_X (-k) \to \bigoplus_{i=1}^{\nu} \mathcal{E}_{-w_i}
\end{equation*}
with
\begin{equation*}
	\mathcal{E}_{-w_i} : = \mathcal{E}_{\le -w_i} / \mathcal{E}_{\le -w_{i-1}}
\end{equation*}
as in \cite[Lemma 4.4.3]{H-L}. The subset $\{ \hat{1} , \dots , \hat{\nu} \}$ consists of the indices $i$ such that $\mathrm{rk} (\mathcal{E}_{-w_i}) >0$  (see \cite[Definition 2.2]{HK1} and also Remark \ref{remdefwalpha}). Defining $X^{reg}$ to be the Zariski open subset of $X$ over which $\mathcal{E}_{\le -w_i}$ are all locally free \cite[Definition 2.4]{HK1}, for each $\alpha \in \{ \hat{1} , \dots , \hat{\nu} \}$ we define $E_{-w_{\alpha}}$ to be equal to the quotient (as a $C^{\infty}$ complex vector bundle) of $\mathcal{E}_{\le -w_{\alpha}} |_{X^{reg}}$ by $\mathcal{E}_{\le -w_{\alpha -1}} |_{X^{reg}}$ (see \cite[discussion following Definition 2.4]{HK1}).

Thus there exists a $C^{\infty}$-isomorphism 
\begin{equation} \label{mdualisoev}
\overline{\mathcal{E}^{\vee}} \stackrel{\sim}{\to}  \bigoplus_{\alpha =\hat{1}}^{\hat{\nu}} E_{-w_{\alpha}}
\end{equation}
of complex vector bundles over $X^{reg}$ (see \cite[(2.8)]{HK1}) such that
\begin{equation} \label{defewt}
	e^{wt} := \mathrm{diag} (e^{ w_{\hat{1}} t} , \cdots , e^{w_{\hat{\nu}} t})
\end{equation}
acts on $\bigoplus_{\alpha =\hat{1}}^{\hat{\nu}} E_{-w_{\alpha}}$, with $e^{w_{\alpha}t }$ acting on the factor $E_{-w_{\alpha}}$. We then define
\begin{equation*}
		\hat{h}_{\sigma_t} :=  e^{- w t} h_{\sigma_t}  e^{- w t} \in   \Gamma_{C^{\infty}_{X^{reg}}} ( \mathcal{E}^{\vee} \otimes  \overline{\mathcal{E}^{\vee}}),
\end{equation*}
which we call the \textbf{renormalised Bergman 1-PS} associated to $\sigma_t$ \cite[Definition 2.9]{HK1}, in which the $C^{\infty}$-isomorphism (\ref{mdualisoev}) is understood. An important fact is that this 1-PS is convergent \cite[Proposition 2.8]{HK1}, and we call the limit
\begin{equation*}
		\hat{h} :=   \lim_{t \to + \infty} e^{- w t} h_{\sigma_t}  e^{- w t} \in  \Gamma_{C^{\infty}_{X^{reg}}} ( \mathcal{E}^{\vee} \otimes  \overline{\mathcal{E}^{\vee}})
	\end{equation*}
the \textbf{renormalised Quot-scheme limit} of $h_{\sigma_t}$ \cite[Definition 2.9]{HK1}, which is positive definite over $X^{reg}$ \cite[Lemma 2.11]{HK1}.

The renormalised Quot-scheme limit of $h_{\sigma_t}$ can be regarded as a differential-geometric analogue of the Quot-scheme limit in algebraic geometry, as explained in \cite[Section 2]{HK1}. The choice of $\{ w_{\alpha} \}_{\alpha =\hat{1}}^{\hat{\nu}}$ in (\ref{defsubswgt}) precisely corresponds to the weights of $\zeta$ on the components of $\bigoplus_{i=1}^{\nu} \mathcal{E}_{-w_i}$ whose rank is nontrivial. From now on in the main body of the text, we shall consistently use the subscript $\alpha$ to denote this particular subset $\{ w_{\alpha} \}_{\alpha =\hat{1}}^{\hat{\nu}}$.

\begin{rem} \label{remdefwalpha}
	The precise meaning of $\{ w_{\alpha} \}_{\alpha =\hat{1}}^{\hat{\nu}} \subset \{ w_i \}_{i=1}^{\nu}$ is as follows: the subscript $\alpha$ runs over a subset $\{ \hat{1} , \dots , \hat{\nu} \}$  of $\{ 1 , \dots , \nu \}$, with the ordering given by $\hat{1} < \hat{2} < \dots < \hat{\nu}$. It turns out that $\hat{1} = 1$ (see \cite[Remark 2.3]{HK1}). The reader is referred to \cite[Section 2]{HK1} for more details.
\end{rem}

\subsection{The non-Archimedean Donaldson functional}

We now recall the non-Archimedean Donaldson functional from \cite{HK1}, whose definition involves the filtration (\ref{satfiltevs}) of $\mathcal{E}$ by saturated subsheaves.

We choose $j(\zeta , k) \in \mathbb{N}$ to be the minimum integer so that 
\begin{equation} \label{defsmjzk}
	j(\zeta , k) w_i \in \mathbb{Z}
\end{equation}
for all $i=1 , \dots, \nu$. Writing $\bar{w}_i:= j(\zeta , k) w_i$, we may replace the filtration (\ref{satfiltevs}) by
\begin{equation} \label{itsatfiltevs}
	0 \neq \mathcal{E}'_{\le - \bar{w}_{1}} \subset \cdots \subset \mathcal{E}'_{\le - \bar{w}_{\nu}} = \mathcal{E}
\end{equation}
which is graded by integers. With this understood, the following was defined in \cite[Definition 4.3]{HK1}.

\begin{definition}
	\textbf{The non-Archimedean Donaldson functional} $\mathcal{M}^{\mathrm{NA}} (\zeta , k)$ is a rational number defined as
	\begin{equation*}
		\mathcal{M}^{\mathrm{NA}} (\zeta , k) := \frac{2}{j(\zeta , k)} \sum_{q \in \mathbb{Z}} \mathrm{rk} (\mathcal{E}'_{\le q}) \left(  \mu(\mathcal{E}) - \mu (\mathcal{E}'_{\le q}) \right).
	\end{equation*}
\end{definition}

An important point is that the positivity of $\mathcal{M}^{\mathrm{NA}} (\zeta , k)$ is equivalent to the slope stability of $\mathcal{E}$, as stated in the following (see also \cite[Section 5]{HK1}).

\begin{prop} \emph{(\cite[Proposition 6.2]{HK1})} \label{prop62hk1}
	The non-Archimedean Donaldson functional $\mathcal{M}^{\mathrm{NA}} (\zeta , k)$ is positive (resp.~nonnegative) for all $\zeta \in \mathfrak{sl} (H^0 (\mathcal{E}(k))^{\vee})$ and $k \ge \mathrm{reg}(\mathcal{E})$ whose associated filtration (\ref{satfiltevs}) is nontrivial, if and only if $\mathcal{E}$ is slope stable (resp.~semistable).
\end{prop}

\begin{rem} \label{remfilttriv}
	In the above, we decreed that a filtration is trivial if it is $0 \subsetneq \mathcal{E}$. Later, we shall define a quantity which detects such triviality (Definition \ref{defjna}).
\end{rem}

We now recall and state the main results of \cite{HK1} as follows.

\begin{thm} \emph{(\cite[Theorem 1]{HK1})} \label{thmlnsdf}
There exists a constant $c_k >0$ that depends only on $h_{\mathrm{ref}}$ and $k \in \mathbb{N}$ such that
\begin{equation*}
		\mathcal{M}^{Don}(h_{\sigma_t} ,  h_{\mathrm{ref}}) \ge  \mathcal{M}^{\mathrm{NA}} (\zeta , k) t - c_{k}
\end{equation*}
holds for all $t \ge 0$ and all $\zeta \in \mathfrak{sl}(H^0(\mathcal{E}(k))^{\vee})$.
\end{thm}

\begin{rem}
While the above theorem is all we need in this paper, we can further show (cf.~\cite[Theorem 2]{HK1}) that
\begin{equation*}
\mathcal{M}^{Don} (h_{\sigma_t}, h_{\mathrm{ref}}) = \mathcal{M}^{\mathrm{NA}} (\zeta , k) t + O(1),
\end{equation*}
where $O(1)$ stands for the term that remains bounded as $t \to + \infty$. In particular, we have (cf.~\cite[Corollary 6.3]{HK1})
\begin{equation*}
\lim_{t \to + \infty} \frac{\mathcal{M}^{Don} (h_{\sigma_t}, h_{\mathrm{ref}})}{t} = \mathcal{M}^{\mathrm{NA}} (\zeta , k),
\end{equation*}
which shows that $\mathcal{M}^{\mathrm{NA}} (\zeta , k)$ is the term that controls the asymptotic behaviour of $\mathcal{M}^{Don}(h_{\sigma_t}, h_{\mathrm{ref}})$.

It may also be worth noting that the analysis used to prove the above results are elementary, cf.~\cite[Section 3]{HK1}.
\end{rem}

\section{Slope stability as uniform stability} \label{ssstaust}

We prove that slope stability can be interpreted as a ``uniform'' stability condition, in terms of \cite{BHJ1,Dertwisted}. The following is the key definition that will be important later.

\begin{definition} \label{defjna}
Suppose that $\zeta \in \mathfrak{sl} (H^0 (\mathcal{E}(k))^{\vee})$ has eigenvalues $w_1 , \dots , w_{\nu} \in \mathbb{Q}$, and let $j(\zeta , k)$ be as defined by (\ref{defsmjzk}). Writing $\bar{w}_{\alpha} :=  j(\zeta , k) w_{\alpha} \in \mathbb{Z}$ for $\alpha = \hat{1} , \dots , \hat{\nu}$, we define
\begin{equation*}
		J^{\mathrm{NA}} (\zeta , k ) := \max_{\hat{1} \le \alpha , \beta \le \hat{\nu}} \frac{|\bar{w}_{\alpha} - \bar{w}_{\beta} |}{j(\zeta , k)} = \max_{\hat{1} \le \alpha , \beta \le \hat{\nu}} |w_{\alpha} - w_{\beta}| \ge 0.
\end{equation*}
\end{definition}

\begin{rem}
	The notation $J^{\mathrm{NA}}$ is chosen simply because the role it plays is analogous to the non-Archimedean $J$-functional \cite{BHJ1} or the minimum norm \cite{Dertwisted} for the case of varieties. In particular, the proof of \cite[Lemma 2.5]{HK1} shows that the filtration (\ref{satfiltevs}) defined by $\zeta$ is nontrivial (in the sense of Proposition \ref{prop62hk1} and Remark \ref{remfilttriv}) if and only if $J^{\mathrm{NA}} (\zeta , k ) > 0$.
	
	Note that we do \textit{not} define a functional $J$ that has $J^{\mathrm{NA}}$ as its slope at infinity (which would be more natural, following \cite{BHJ1,BHJ2}). On the other hand, it is worth pointing out that $J^{\mathrm{NA}} (\zeta , k )$ is defined in terms of purely algebro-geometric data, as the maximum difference of the weights on the non-torsion components of the Quot-scheme limit, which a priori has nothing to do with hermitian metrics.
\end{rem}

\begin{rem} \label{jnascalinv}
	Note that while $J^{\mathrm{NA}} (\zeta , k )$ is defined only for $\zeta \in \mathfrak{sl} (H^0 (\mathcal{E}(k))^{\vee})$, it can be naturally extended to $\mathfrak{gl} (H^0 (\mathcal{E}(k))^{\vee})$ since it is invariant under the constant rescaling $\zeta \mapsto \zeta + c \mathrm{Id}$.
\end{rem}

If $\mathcal{E}$ is slope stable, we show that in fact there is a strict lower bound for $\mathcal{M}^{\mathrm{NA}}$ in terms of $J^{\mathrm{NA}} (\zeta , k )$. We start with the following observation.

\begin{lem}\label{uniflem}
	Suppose that $\mathcal{E}$ is slope stable. Then there exists a constant $c_{\mathcal{E}} >0$ such that for any coherent subsheaf $\mathcal{F} \subset \mathcal{E}$ with $0<\mathrm{rk}(\mathcal{F}) < \mathrm{rk}(\mathcal{E})$ we have 
		\begin{equation*}
		\frac{\mathrm{deg}(\mathcal{E})}{\mathrm{rk}(\mathcal{E})} - \frac{\mathrm{deg}(\mathcal{F})}{\mathrm{rk}(\mathcal{F})} \ge c_{\mathcal{E}} >0.
	\end{equation*}
\end{lem}

\begin{proof}
	This simply follows from the fact that the degree and the rank are both integers (since $X$ is smooth), with $0<\mathrm{rk}(\mathcal{F}) < \mathrm{rk}(\mathcal{E})$.
\end{proof}

This implies the following lower bound for $\mathcal{M}^{\mathrm{NA}}$ which is crucially important in our proof of the Donaldson--Uhlenbeck--Yau theorem.

\begin{prop}\label{uniflemmna}
	Suppose that $\mathcal{E}$ is slope stable. Then we have 
	\begin{equation}
		\mathcal{M}^{\mathrm{NA}} (\zeta , k) \ge  2c_{\mathcal{E}}  \cdot J^{\mathrm{NA}} (\zeta , k ) , \label{unifineq}
	\end{equation}
	with $c_{\mathcal{E}}>0$ as in Lemma \ref{uniflem}.
\end{prop}

\begin{proof}
	Recalling the definition of $\{ w_{\alpha} \}_{\alpha = \hat{1}}^{\hat{\nu}}$ as given in (\ref{defsubswgt}), \cite[Lemma 2.5]{HK1} implies that $0< \mathrm{rk} (\mathcal{E}'_{ -w_i}) \le r$ if and only if $i \in \{ \hat{1} , \dots , \hat{\nu}\}$, and $\mathrm{rk} (\mathcal{E}'_{\le -w_i}) = r$ if and only if $i \ge \hat{\nu}$. In particular, $\mathrm{rk} (\mathcal{E}'_{\le q}) = 0$ if $q < - \bar{w}_{\hat{1}} $ and $\mu (\mathcal{E}'_{\le q}) = \mu (\mathcal{E}) $ if $- \bar{w}_{\hat{\nu}} \le q $ (by noting that $\mathcal{E}'_{\le q}$ is saturated in $\mathcal{E}$). We thus get
\begin{equation*}
 \sum_{q \in \mathbb{Z}} \mathrm{rk} (\mathcal{E}_{\le q}) \left( \mu(\mathcal{E}) - \mu (\mathcal{E}'_{\le q}) \right) = \sum_{q = - \bar{w}_{\hat{1}}}^{- \bar{w}_{\hat{\nu} }-1} \mathrm{rk} (\mathcal{E}'_{\le  q}) \left(  \mu(\mathcal{E}) - \mu (\mathcal{E}'_{\le q}) \right) .
\end{equation*}

Thus, combined with Lemma \ref{uniflem}, we get
\begin{align*}
	\mathcal{M}^{\mathrm{NA}} (\zeta , k) &= \frac{2}{j (\zeta , k)}\sum_{q = - \bar{w}_{\hat{1}}}^{- \bar{w}_{\hat{\nu} }-1} \mathrm{rk} (\mathcal{E}'_{\le  q}) \left(  \mu(\mathcal{E}) - \mu (\mathcal{E}'_{\le q}) \right)\\
	&= \frac{2}{j (\zeta , k)}\sum_{q = - \bar{w}_{\hat{1}}}^{- \bar{w}_{\hat{\nu} }-1} \mathrm{rk} (\mathcal{E}'_{\le q}) \left( \frac{\mathrm{deg}(\mathcal{E})}{\mathrm{rk}(\mathcal{E})} - \frac{\mathrm{deg}(\mathcal{E}'_{\le q})}{\mathrm{rk}(\mathcal{E}'_{<q})} \right) \\
	&\ge \frac{2}{j(\zeta , k)} \sum_{q = - \bar{w}_{\hat{1}}}^{- \bar{w}_{\hat{\nu} }-1} \mathrm{rk} (\mathcal{E}'_{\le q}) c_{\mathcal{E}} \\
	&\ge 2 c_{\mathcal{E}} \frac{\bar{w}_{\hat{1}} - \bar{w}_{\hat{\nu}} }{j(\zeta , k)}.
\end{align*}
By recalling the ordering $w_1 = w_{\hat{1}} > \cdots > w_{\hat{\nu}} > \cdots > w_{\nu}$, as in (\ref{ordlambda}) and Remark \ref{remdefwalpha}, we get the result.
\end{proof}
\begin{rem}
Observe that $\mathcal{M}^{\mathrm{NA}} (\zeta , k)$ and $J^{\mathrm{NA}} (\zeta , k )$ are both equal to zero for all $\zeta$ and $k$ if $\mathcal{E}$ is a line bundle, which fundamentally comes from the fact that the stability condition (Definition \ref{defmtstab}) is vacuous for line bundles (i.e.~there exists no subsheaf $\mathcal{F} \subset \mathcal{E}$ with $0 < \mathrm{rk} (\mathcal{F}) < \mathrm{rk} (\mathcal{E})$ if $\mathcal{E}$ is a line bundle). In particular, Proposition \ref{uniflemmna} provides no nontrivial information for line bundles.
\end{rem}

\section{From slope stability to Hermitian--Einstein metrics}\label{stabtocan}

\subsection{Uniform coercivity of the Donaldson functional}
Suppose that we fix a reference hermitian metric $h_{\mathrm{ref}} \in \mathcal{H}_{\infty}$, and pick a sequence $\{ h_k \}_{k \in \mathbb{N}} \subset \mathcal{H}_{\infty}$, $h_k \in \mathcal{H}_k$, such that $h_k \to h_{\mathrm{ref}}$ in the $C^{p}$-norm as $k \to \infty$. In what follows, we take $p$ to be an integer with $p \ge 2$.

Pick $\zeta \in \mathfrak{sl}(H^0(\mathcal{E}(k))^{\vee})$ and define $h_{\sigma_t}$ to be the Bergman 1-PS emanating from $h_k$ generated by $\zeta$. In Theorem \ref{thmlnsdf} (or \cite[Theorem 1]{HK1}), we proved that there exists a constant $c_k >0$ that depends on $k$ such that
\begin{equation*}
	\mathcal{M}^{Don} (h_{\sigma_t},h_{\mathrm{ref}}) \ge \mathcal{M}^{\mathrm{NA}}(\zeta , k) t - c_k
\end{equation*}
uniformly for all $t \ge 0$ and $\zeta \in \mathfrak{sl}(H^0(\mathcal{E}(k))^{\vee})$. It seems natural to conjecture that the above inequality can be strengthened as follows (see also \cite[Remark 3.10]{HK1}).
 
 \begin{hypothesis} \label{conjpczero}
 There exists a constant $c_{\mathrm{ref}} >0$ which depends only on the reference metric $h_{\mathrm{ref}} \in \mathcal{H}_{\infty}$ such that
 \begin{equation} \label{pczero}
	\mathcal{M}^{Don} (h_{\sigma_t},h_{\mathrm{ref}}) \ge  \mathcal{M}^{\mathrm{NA}}(\zeta , k) t - c_{\mathrm{ref}} 
\end{equation}
uniformly for all $t \ge 0$, $\zeta \in \mathfrak{sl}(H^0(\mathcal{E}(k))^{\vee})$, and $k \in \mathbb{N}$.
 \end{hypothesis}

It is tempting to point out an analogy with the case for the K\"ahler--Einstein metrics, in which a similar inequality was achieved by establishing the partial $C^0$-estimate (\cite{DonSun,Szepartial}; see also \cite[Section 6]{BHJ2}). For the vector bundles, a natural statement for the partial $C^0$-estimate may be the following. Let $\{ h_t \}_{t \ge 0}$ be the solution to the Yang--Mills flow starting at $h_0$ and let $\Vert \cdot \Vert$ be an appropriate Sobolev norm. Then the partial $C^0$-estimate would claim that for all $\epsilon >0$ there exists $k = k(\epsilon) \in \mathbb{N}$ such that for each $h_t$ there exists a Fubini--Study metric $h'_t \in \mathcal{H}_k$ at level $k$ such that $\Vert h_t - h'_t \Vert < \epsilon$ for all $t \ge 0$; the crucial part is that $k$ can be chosen uniformly for all $t$. This, together with Theorem \ref{thmlnsdf}, will certainly imply Hypothesis \ref{conjpczero} along the Yang--Mills flow.

\begin{rem}
Note that similar assumptions were made by Paul for the case of constant scalar curvature K\"ahler metrics or K\"ahler--Einstein metrics: see \cite[Conjecture 5.1]{Paul2012cm} and \cite[Corollary 1.6]{Paul13}.	
\end{rem}

Assuming the truth of Hypothesis \ref{conjpczero} gives us the following immediate consequence for a slope stable bundle $\mathcal{E}$.

\begin{prop} \label{pcorrdonna}
	Suppose that Hypothesis \ref{conjpczero} is true and that $\mathcal{E}$ is slope stable. Then, for any rational Bergman 1-PS $\{ h_{\sigma_t} \}_{t \ge 0}$ we have
\begin{equation} \label{corrdonna}
	\mathcal{M}^{Don} (h_{\sigma_t} , h_{\mathrm{ref}}) \ge  2 c_{\mathcal{E}} \cdot J^{\mathrm{NA}} (\zeta , k) t -c_{\mathrm{ref}}.
\end{equation} 
In particular, $\mathcal{M}^{Don}(-, h_{\mathrm{ref}})$ is bounded from below on the space $\mathcal{H}_{\infty}$ of all smooth hermitian metrics. 
\end{prop}

\begin{proof}
	The inequality (\ref{corrdonna}) is immediate from Proposition \ref{uniflemmna}. Since $\zeta \in \mathfrak{sl}(H^0(\mathcal{E}(k))^{\vee})$ with rational eigenvalues are dense in $\mathfrak{sl}(H^0(\mathcal{E}(k))^{\vee})$ and $\mathcal{M}^{Don}$ is invariant under scaling (\ref{mdsclinv}), Corollary \ref{corthmbergexp} (afforded by Theorem \ref{thmbergexp}) immediately implies that $\mathcal{M}^{Don}$ is bounded from below on $\mathcal{H}_{\infty}$.
\end{proof}

\subsection{Proof of Theorem \ref{mainthmhk}}

In order to study variational properties of the Donaldson functional, we need to consider a completion of $\mathcal{H}_{\infty}$ by the $C^p$-norm, as defined below.

\begin{definition} \label{Hp-space}
For $p \in \mathbb{N}$, $p \ge 2$, we define $\mathcal{H}_{[p]}$ to be the space of hermitian metrics on $\mathcal{E}$ that is of class $C^p$, with the topology induced from the $C^p$-norm (with respect to a fixed hermitian metric).

We further fix the scaling as follows. Fixing $h_{\mathrm{ref}} \in \mathcal{H}_{\infty}$, we assume in what follows that
	\begin{equation} \label{cvsclngh}
		\inf_{x \in X} \{ \text{least eigenvalue of } h h^{-1}_{\mathrm{ref}} \text{ at } x \} =1
	\end{equation}
	for all $h \in \mathcal{H}_{[p]}$; in particular, no sequence in $\mathcal{H}_{[p]}$ converges to a degenerate hermitian metric.
\end{definition}

Note the obvious inclusions $\mathcal{H}_{[p']} \subset \mathcal{H}_{[p]}$ for $p' \ge p$, and $\mathcal{H}_{\infty}\subset \mathcal{H}_{[p]}$ for all $p \in \mathbb{N}$ (with appropriate scaling as in (\ref{cvsclngh})). Unlike the case of varieties, it turns out that such a classical $C^p$-completion suffices for our purpose; this is perhaps related to the geodesic completeness of $\mathcal{H}_{\infty}$ (or $\mathcal{H}_{[p]}$, as stated in the lemma below).

 We prove some straightforward results concerning $\mathcal{H}_{[p]}$, which can be proved entirely analogously to what was proved for $\mathcal{H}_{\infty}$ in Section \ref{defMDon}.

\begin{lem} \label{lemhpprop}
	Fixing a reference metric $h_{\mathrm{ref}} \in \mathcal{H}_{\infty}$ and writing $\mathcal{M}^{Don}$ for $\mathcal{M}^{Don} (- ,h_{\mathrm{ref}})$, we have the following for any $p\geq 2$:
	\begin{enumerate}
		\item $\mathcal{M}^{Don}$ is well-defined and continuous on $\mathcal{H}_{[p]}$;
		\item $h \in \mathcal{H}_{[p]}$ attains the minimum of $\mathcal{M}^{Don}$ over $\mathcal{H}_{[p]}$ if and only if it satisfies $$\Lambda_\omega F_h = \frac{\mu(\mathcal{E})}{\Vol_L} \Id_{\mathcal{E}},$$ which is well-defined;
		\item the critical point of $\mathcal{M}^{Don}$ on $\mathcal{H}_{[p]}$, if exists, is unique if $\mathcal{E}$ is simple;
		\item $\mathcal{M}^{Don}$ is bounded from below on $\mathcal{H}_{\infty}$ if and only if it is so on $\mathcal{H}_{[p]}$;
		\item any $h \in \mathcal{H}_{[p]}$ can be connected to $h_{\mathrm{ref}}$ by a geodesic.
	\end{enumerate}
\end{lem}

\begin{proof}
	The first item is straightforward. The second can be proved as in Lemma \ref{derivmdon}. The third follow from Proposition \ref{lemmdonconv} proved in the appendix. The fourth follows from the continuity of $\mathcal{M}^{Don}$. The fifth can be proved by explicitly writing down the geodesic $\{ h_s \}_{0 \le s \le 1}$ as $h_s = \exp( s \log h h^{-1}_{\mathrm{ref}}) h_{\mathrm{ref}}$, where we note that the scaling convention (\ref{cvsclngh}) is preserved for all $0 \le s \le 1$ (recall also the remark after Definition \ref{defgeod}).
\end{proof}

From now,  our aim is to prove that the Donaldson functional achieves the minimum over the space $\mathcal{H}_{[p]}$ if $\mathcal{E}$ is slope stable. Throughout in what follows, we pick and fix some $p \ge 2$ and a reference hermitian metric $h_{\mathrm{ref}} \in \mathcal{H}_{\infty}$ once and for all, and consider $\mathcal{M}^{Don}:=\mathcal{M}^{Don}(-,h_{\mathrm{ref}})$ to be defined over $\mathcal{H}_{[p]}$.

Observe first that $\mathcal{E}$ is simple by Lemma \ref{stablesimple}, and hence the critical point of $\mathcal{M}^{Don}$ is the unique minimum by Proposition \ref{lemmdonconv}. Theorem \ref{thmbergexp} implies that there exists a sequence $h_k \in \mathcal{H}_k$ of Fubini--Study metrics that converge to $h_{\mathrm{ref}}$ in $C^{2p+2}$ for any $p\geq 2$. We define the reference metric at each $k$ to be the one defined by $h_k := Q^*(k) Q(k)$.

Before stating and proving Proposition \ref{pmdproperhp}, which is the main result of this section, we prove the following rather technical lemma that we need in its proof.

\begin{lem} \label{lmdproperhp}
	Suppose that we have a sequence $\{ h_{t_i \zeta_i} \}_{i \in \mathbb{N}}$, where
	\begin{equation*}
		h_{t_i \zeta_i} := Q^*(k_i) e^{\zeta_i^* t_i} e^{\zeta_i t_i } Q (k_i) \in \mathcal{H}_{k_i}
	\end{equation*}
	with $\zeta_i \in \mathfrak{gl} (H^0 (\mathcal{E} (k_i))^{\vee})$. Suppose also that each $h_{t_i \zeta_i}$ satisfies the scaling convention (\ref{cvsclngh}). For any fixed constant $\epsilon_J \in (0,1/4)$ the following holds: for all $i \in \mathbb{N}$ there exists $\xi_i \in \mathfrak{gl} (H^0 (\mathcal{E} (k_i))^{\vee})$ which satisfies
	\begin{enumerate}
		\item $\Vert \xi_i - \zeta_i \Vert_{op} \le 1$;
		\item $J^{\mathrm{NA}}(\xi_i , k_i) \ge \epsilon_J$;
		\item $\mathcal{M}^{Don} (h_{t_i \zeta_i}) = \mathcal{M}^{Don} (h_{t_i \xi_i})$;
		\item $h_{t_i \xi_i}$ satisfies the scaling convention (\ref{cvsclngh}).
	\end{enumerate}
\end{lem}

\begin{proof}
	For each $\zeta_i$ we shall construct a path $\{ \xi_{i, \tau} \}_{0 \le \tau \le 1}$ in $\mathfrak{gl}(H^0(\mathcal{E}(k_i))^{\vee})$ such that $\xi_{i,0} = \zeta_i$ and $\xi_{i,1}$ satisfies the first three properties in the statement.

For each $\eta \in \mathfrak{gl}(H^0(\mathcal{E}(k_i))^{\vee})$ we define $h_{\eta}$ to be $Q^*(k_i) e^{\eta^*} e^{\eta}Q(k_i)$, and view the Donaldson functional as a map defined on $\mathfrak{gl}(H^0(\mathcal{E}(k_i))^{\vee})$ by
	\begin{equation*}
		\mathcal{M}^{Don} :  \mathfrak{gl}(H^0(\mathcal{E}(k_i))^{\vee}) \ni \eta \mapsto \mathcal{M}^{Don} (h_{\eta}) \in \mathbb{R} ,
	\end{equation*}
	which implies that its derivative $\delta \mathcal{M}^{Don} |_{\eta}$ at $\eta$ is a linear map
	\begin{equation*}
		\delta \mathcal{M}^{Don} |_{\eta} : \mathfrak{gl}(H^0(\mathcal{E}(k_i))^{\vee}) \ni \tilde{\eta} \mapsto \delta \mathcal{M}^{Don} |_{\eta} (\tilde{\eta}) \in \mathbb{R} .
	\end{equation*}
	We further restrict the domain of $\delta \mathcal{M}^{Don} |_{\eta}$ to the set of hermitian matrices, and consider $\tilde{\eta}$ that is hermitian. Since $\delta \mathcal{M}^{Don} |_{\eta}$ is a linear map in $\tilde{\eta}$ of (real) rank 1, the real dimension of its kernel is $N_{k_i} (N_{k_i}+1)/2 -1$. If $k_i$ is large enough for all $i$, we can thus define a smooth vector field on $\mathfrak{gl}(H^0(\mathcal{E}(k_i))^{\vee})$ in such a way that it defines a tangent vector $\tilde{\eta}$ at $\eta$ with the properties that
	\begin{itemize}
		\item $\tilde{\eta}$ commutes with $\eta$ (thus we may assume that $\tilde{\eta}$ and $\eta$ are both diagonal), and 
		\item $J^{\mathrm{NA}}(\eta + a_{\eta} \tilde{\eta} , k_i) > J^{\mathrm{NA}}(\eta , k_i)$ if $a_{\eta} >0$ is small enough.
	\end{itemize}
	By applying the cutoff function, we may further assume that the vector field is compactly supported in the region
	\begin{equation*}
		\{ \xi \in \mathfrak{sl}(H^0(\mathcal{E}(k_i))^{\vee}) \mid  \Vert \xi - \zeta_i \Vert_{op} \le 2 \epsilon_J \} .
	\end{equation*}
	Observe also that, given a smooth vector field that is supported on the above compact region of $\mathfrak{gl}(H^0(\mathcal{E}(k_i))^{\vee})$, we can always construct its integral curve emanating from $\zeta_i$. Thus we can construct a path $\{ \xi_{i, \tau} \}_{0 \le \tau \le 1}$, with $\xi_{i,0} = \zeta_i$ and $\Vert \xi_{i, \tau} - \zeta_i \Vert_{op} \le 2 \epsilon_J$ for $0 \le \tau \le 1$, so that $J^{\mathrm{NA}}(\xi_{i, \tau} , k_i)$ is monotonically increasing in $\tau$. Moreover, writing $w_{\hat{1}}(i, \tau ), \dots , w_{\hat{\nu}} (i, \tau )$ for the weights of $\xi_{i, \tau}$ (as defined in (\ref{defsubswgt})) and recalling $J^{\mathrm{NA}}(\xi_{i, \tau} , k_i) = w_{\hat{1}} (i, \tau) - w_{\hat{\nu}} (i, \tau)$ (Definition \ref{defjna}), the above argument means that either $w_{\hat{1}}(i, \tau )$ increases or $w_{\hat{\nu}}(i, \tau )$ decreases (or both) as $\tau$ increases (if $\hat{\nu} = \hat{1}$ we simply choose the tangent vector at $\xi_{i,0} = \zeta_i$ to split the $w_{\hat{1}} (i , 0)$-eigenspace and continue), which then implies that $J^{\mathrm{NA}}(\xi_{i, \tau} , k_i)$ increases at least by the above increment in $w_{\hat{1}}(i, \tau )$ or $w_{\hat{\nu}}(i, \tau )$, as $\tau$ increases.

	Thus, we can construct a path $\{ \xi_{i, \tau} \}_{0 \le \tau \le 1}$ such that its endpoint $\xi_{i,1} \in \mathfrak{gl}(H^0(\mathcal{E}(k_i))^{\vee})$ satisfies
	\begin{itemize}
		\item $\Vert \xi_{i,1} - \zeta_i \Vert_{op} \le 2 \epsilon_J$,
		\item $J^{\mathrm{NA}}(\xi_{i,1} , k_i) \ge \epsilon_J$, and
		\item $\mathcal{M}^{Don} (h_{t_i \zeta_i}) = \mathcal{M}^{Don} (h_{t_i \xi_{i,1}})$.
	\end{itemize}
	Finally, note that we may add a constant multiple of the identity to $\xi_{i,1}$ so that $\xi_i := \xi_{i,1} + c_{\xi} I$ satisfies the fourth property stated in the lemma, i.e.~$h_{t_i \xi_i}$ satisfies the scaling convention (\ref{cvsclngh}). Since $h_{t_i \zeta_i}$ satisfies (\ref{cvsclngh}) and $\Vert \xi_{i,1} - \zeta_i \Vert_{op} \le 2 \epsilon_J$, we have $| c_{\xi} | \le 2 \epsilon_J$ and hence $\Vert \xi_{i} - \zeta_i \Vert_{op} \le 4 \epsilon_J \le 1$. Recalling $J^{\mathrm{NA}}(\xi_{i,1} , k_i) = J^{\mathrm{NA}}(\xi_{i} , k_i)$ by Remark \ref{jnascalinv} and $\mathcal{M}^{Don} (h_{t_i \xi_{i,1}}) = \mathcal{M}^{Don} (h_{t_i \xi_{i}})$ by (\ref{mdsclinv}), we establish all the four conditions stated in the lemma.
\end{proof}

The following is the main technical result of this section.
 
\begin{prop} \label{pmdproperhp}
	Suppose that Hypothesis \ref{conjpczero} holds and that $\mathcal{E}$ is slope stable. Then there exists $h_{\mathrm{min}} \in \mathcal{H}_{[p]}$ which attains the minimum of $\mathcal{M}^{Don}$.
\end{prop}

\begin{proof}
	As we saw in Proposition \ref{pcorrdonna}, Hypothesis \ref{conjpczero} and slope stability of $\mathcal{E}$ implies that $\mathcal{M}^{Don}$ is bounded below over $\mathcal{H}_{\infty}$. We pick $a>0$ so that the interval $[-a, a]$ contains the infimum of $\mathcal{M}^{Don}$ over $\mathcal{H}_{\infty}$. Let $$\mathcal{A}:= \left(\mathcal{M}^{Don}\right)^{-1} ([-a, a]) \cap \mathcal{H}_{\infty}$$ be the preimage of the interval $[-a, a]$ under $\mathcal{M}^{Don}$.

We fix some notation. Pick $\epsilon >0$ to be sufficiently small, so that $2a> a+ \epsilon$. Given a minimising sequence $\{ h_i \}_{i \in \mathbb{N}} \subset \mathcal{A}$ for $\mathcal{M}^{Don}$, Corollary \ref{corthmbergexp}, afforded by Theorem \ref{thmbergexp}, implies that for each $h_i$ there exists a Fubini--Study metric $\tilde{h}_i \in  \mathcal{H}_{k_i}$ such that
\begin{equation*}
	\left\Vert \tilde{h}_i - h_i \right\Vert_{C^{p+1}} + \left| \mathcal{M}^{Don}(h_i) - \mathcal{M}^{Don}(\tilde{h}_i) \right| < \epsilon,
\end{equation*}
by taking $k_i \in \mathbb{N}$ to be sufficiently large. Then, writing 
\begin{equation*}
	\mathcal{A}_{\epsilon} := \left(\mathcal{M}^{Don}\right)^{-1} ([-a- \epsilon , a+ \epsilon ]) \cap \mathcal{H}_{\infty} ,
\end{equation*}
we get a sequence $\{ \tilde{h}_i \}_{i \in \mathbb{N}}$ in $\mathcal{A}_{\epsilon}$ of Fubini--Study metrics (with $\tilde{h}_i \in \mathcal{H}_{k_i}$), which approximates the sequence $\{ h_i \}_{i \in \mathbb{N}}$. By choosing each $k_i$ to be large enough, we may assume that the sequence $\{ \tilde{h}_i \}_{i \in \mathbb{N}}$ itself is a minimising sequence for $\mathcal{M}^{Don}$. Recalling the scale invariance of $\mathcal{M}^{Don}$ (\ref{mdsclinv}), we may further assume that for all $i \in \mathbb{N}$ the metrics $h_i$ and $\tilde{h}_i$ satisfy the scaling convention (\ref{cvsclngh}).

Our aim in what follows is to prove that there exists a minimising sequence for $\mathcal{M}^{Don}$ that contains a convergent subsequence in $\mathcal{H}_{[p]}$.

Suppose not. Then no minimising sequence for $\mathcal{M}^{Don}$ contains a convergent subsequence in $\mathcal{H}_{[p]}$. Pick an arbitrary minimising sequence $\{ h_i \}_{i \in \mathbb{N}} \subset \mathcal{A}$. By taking $k_i$'s to be large enough, we may further assume that its Fubini--Study approximation $\{ \tilde{h}_i \}_{i \in \mathbb{N}}$ (which also defines a minimising sequence for $\mathcal{M}^{Don}$) contains no convergent subsequence in the $C^p$-norm either. Thus, the Arzel\`a--Ascoli theorem implies that the $C^{p+1}$-norm of $\tilde{h}_i$ cannot be bounded.

We may write $$\tilde{h}_i = Q^* ( k_i) e^{\zeta^*_i t_i} e^{\zeta_i t_i} Q(k_i),$$ where $\zeta_i \in \mathfrak{gl}(H^0(\mathcal{E}(k_i))^{\vee})$ and $\Vert \zeta_i \Vert_{op} =1$. Further, by using Lemma \ref{lmdproperhp}, we may replace $\zeta_i$ by $\xi_i$ for each $i$, so that 
\begin{equation*}
	\tilde{h}'_i := Q^* ( k_i) e^{\xi^*_i t_i} e^{\xi_i t_i} Q(k_i) ,
\end{equation*}
with $J^{\mathrm{NA}} (\xi_i ,k_i ) \ge \epsilon_J$, is still a minimising sequence for the Donaldson functional; in particular $\{ \tilde{h}'_i \}_{i \in \mathbb{N}} \subset \mathcal{A}_{\epsilon}$. Observe that Lemma \ref{lmdproperhp} implies that each $\tilde{h}'_i$ satisfies the scaling convention (\ref{cvsclngh}). Recalling our original hypothesis that no minimising sequence contains a convergent subsequence, $\{ \tilde{h}'_i \}_{i \in \mathbb{N}}$ contains no convergent subsequence in the $C^p$-norm, and hence the $C^{p+1}$-norm of $\tilde{h}'_i$ cannot be bounded.

Recalling that $\{ h_k \}_{k \in \mathbb{N}}$, with $h_{k_i} := Q^*(k_i)Q(k_i)$, is assumed to converge to $h_{\mathrm{ref}}$ in the $C^{2p+2}$-norm, Proposition \ref{parczestzeta} implies that $\Vert \xi_i t_i \Vert_{op}$ cannot remain bounded as $i \to + \infty$. Recalling the convention $\Vert \zeta_i \Vert_{op} \le 1$ from Section \ref{revqslimfs} (or \cite[Remark 2.1]{HK1}) and $\Vert \xi_i - \zeta_i \Vert_{op} \le 1$ (which together imply $\Vert \xi_i \Vert_{op} \le 2$), this in particular implies that the sequence $\{ t_i \}_{i \in \mathbb{N}} \subset \mathbb{R}_{\ge 0}$ is unbounded. Thus, by taking a subsequence if necessary, we may assume that $t_i$ increases monotonically to $+ \infty$ as $i \to  + \infty$. Now, Proposition \ref{pcorrdonna} and Lemma \ref{lmdproperhp} (see also Remark \ref{remscalslgl} concerning the scaling) imply that
\begin{align*}
	\mathcal{M}^{Don}(\tilde{h}'_i) &\ge \mathcal{M}^{\mathrm{NA}} (\xi_i ,k_i ) t_i - c_{\mathrm{ref}} \\
	&\ge 2 c_{\mathcal{E}} J^{\mathrm{NA}} (\xi_i ,k_i ) t_i - c_{\mathrm{ref}} \\
	&\ge 2 c_{\mathcal{E}} \epsilon_J t_i - c_{\mathrm{ref}} \to + \infty
\end{align*}
as $i \to + \infty$; note that the inequality
\begin{equation*}
	\mathcal{M}^{\mathrm{NA}} (\xi_i ,k_i ) \ge 2 c_{\mathcal{E}} J^{\mathrm{NA}} (\xi_i ,k_i ) > 0
\end{equation*}
which follows from the slope stability of $\mathcal{E}$ by Proposition \ref{prop62hk1}, is crucial in the above. Hence
\begin{equation*}
	\mathcal{M}^{Don} (\tilde{h}'_i) > 2 a> a+ \epsilon 
\end{equation*}
for all large enough $i$, contradicting $\{ \tilde{h}'_i \}_{i \in \mathbb{N}} \subset  \mathcal{A}_{\epsilon}$.

Summarising our discussion above, our original assumption that no minimising sequence for $\mathcal{M}^{Don}$ contains a convergent subsequence in $\mathcal{H}_{[p]}$ leads to a contradiction, and hence there must exist a minimising sequence for $\mathcal{M}^{Don}$ which contains a convergent subsequence in $\mathcal{H}_{[p]}$. We write its limit as $h_{\mathrm{min}}$, which is a well-defined hermitian metric in $\mathcal{H}_{[p]}$ by the scaling convention (\ref{cvsclngh}) and necessarily attains the minimum of the Donaldson functional.
\end{proof}

\begin{rem} \label{remscalslgl}
	In the above proof, we wrote $\mathcal{M}^{\mathrm{NA}} (\xi_i ,k_i )$ for $\xi_i \in \mathfrak{gl} (H^0(\mathcal{E}(k_i))^{\vee})$ and used the results in Sections \ref{scrmrhk1} and \ref{ssstaust} for $\xi_i \in \mathfrak{gl} (H^0(\mathcal{E}(k_i))^{\vee})$, although strictly speaking these results were stated only for $\mathfrak{sl} (H^0(\mathcal{E}(k_i))^{\vee})$. This is not significant, since we may perform rescaling as $\xi_i \mapsto \xi_i + c \mathrm{Id}$ ($c \in \mathbb{R}$) so that $\xi_i + c \mathrm{Id} \in \mathfrak{sl} (H^0(\mathcal{E}(k_i))^{\vee})$, by noting that $\mathcal{M}^{Don}$ and $J^{\mathrm{NA}}$ are both invariant under an overall rescaling (see (\ref{mdsclinv}) and Remark \ref{jnascalinv}).
\end{rem}


It remains to show that $h_{\mathrm{min}}$, which attains the minimum of $\mathcal{M}^{Don}$, is the smooth Hermitian--Einstein metric. More precisely, we prove the following.

\begin{prop} \label{propregmin}
	Suppose that $h_{\mathrm{min}} \in \mathcal{H}_{[p]}$ attains the minimum of $\mathcal{M}^{Don}$ over $\mathcal{H}_{[p]}$, and that $\mathcal{E}$ is simple. Then $h_{\mathrm{min}}$ is unique in $\mathcal{H}_{[p]}$. Moreover, $h_{\mathrm{min}}$ is smooth and satisfies
	\begin{equation*}
		\Lambda_{\omega} F_{h_{\mathrm{min}}} = \frac{\mu(\mathcal{E})}{\Vol_L} \Id_{\mathcal{E}}.
	\end{equation*}
	In other words, $h_{\mathrm{min}}$ is the unique minimiser of $\mathcal{M}^{Don}$ in $\mathcal{H}_{\infty}$, which is the well-defined Hermitian--Einstein metric.
\end{prop}

\begin{proof}
That $h_{\mathrm{min}}$ satisfies the Hermitian--Einstein equation, together with its uniqueness in $\mathcal{H}_{[p]}$, follows from Lemma \ref{lemhpprop}. We use the uniqueness to prove that $h_{\mathrm{min}} \in \mathcal{H}_{[p]}$ is in fact smooth. Suppose that, for each choice of $p$, we write $h_{\mathrm{min},p}$ for the minimiser of $\mathcal{M}^{Don}$ in $\mathcal{H}_{[p]}$. The uniqueness of $h_{\mathrm{min},p} \in \mathcal{H}_{[p]}$, together with $\mathcal{H}_{[p+1]} \subset \mathcal{H}_{[p]}$ for all $p \ge 2$, implies $h_{\mathrm{min},p} = h_{\mathrm{min},p+1} = h_{\mathrm{min}, p+2} = \cdots$, by observing that $h_{\mathrm{min},p+1}$ is a critical point of $\mathcal{M}^{Don}$ in $\mathcal{H}_{[p]}$ as well (as it satisfies the Hermitian--Einstein equation). Thus $h_{\mathrm{min}}$ is smooth and minimises $\mathcal{M}^{Don}$ over $\mathcal{H}_{\infty}$, which is necessarily the unique Hermitian--Einstein metric by Lemma \ref{derivmdon} and Proposition \ref{lemmdonconvsm}.
\end{proof}

We have thus obtained the following main result of this paper.
\begin{thm} \label{mainthmss}
	Suppose that Hypothesis \ref{conjpczero} is true. Then $\mathcal{E}$ admits a Hermitian--Einstein metric if it is slope stable. Moreover, the analysis that we need in the proof is elementary, except for the asymptotic expansion of the Bergman kernel (Theorem \ref{thmbergexp}).
\end{thm}

\begin{proof}
	This follows from Propositions \ref{pmdproperhp} and \ref{propregmin}.
\end{proof}

\begin{proof}[Proof of Theorem \ref{mainthmhk}]
	We only need to deal with the case where $\mathcal{E}$ is slope polystable but not stable. In such case we have a holomorphic splitting $\mathcal{E} =\oplus_j \mathcal{E}_j$, with each $\mathcal{E}_j$ being a slope stable (and hence irreducible) bundle. We apply Theorem \ref{mainthmss} to each irreducible component $\mathcal{E}_j$, which provides us with the Hermitian--Einstein metric on each $\mathcal{E}_j$ with Einstein constant $\mu(\mathcal{E}_j)$. Since $\mathcal{E}$ is slope polystable the slopes of the $\mathcal{E}_j$ are all equal, and hence the direct sum of Hermitian--Einstein metrics on $\mathcal{E}_j$ gives the Hermitian--Einstein metric on $\mathcal{E}$.
\end{proof}

\appendix

\section{Convexity of the Donaldson functional}

We review a proof of the well-known theorem that the Donaldson functional is convex along geodesics in $\mathcal{H}_{\infty}$. The aim of presenting the proof of such a well-known result is firstly to make explicit that the same argument carries over to the $\mathcal{H}_{[p]}$ version of it as stated in Proposition \ref{lemmdonconv}, and secondly to provide various formulae that will be useful later in the proof of Theorem \ref{propUnif}.

\begin{prop} \textup{(cf.~\cite[Section 6.3]{Kobook})} \label{lemmdonconvH}
 The functional $\mathcal{M}^{Don}$ is convex along geodesics in $\mathcal{H}_{\infty}$, and its critical point attains the global minimum. Moreover, $\mathcal{M}^{Don}$ is strictly convex along nontrivial geodesics if $\mathcal{E}$ is irreducible (in particular if $\mathcal{E}$ is simple) and in this case the critical point is unique up to an overall constant scaling if it exists.
\end{prop}

\begin{proof}
Let $\{ h_t \}_t \subset \mathcal{H}_{\infty}$ be a path of hermitian metrics on $\mathcal{E}$ parametrised by $t \in (- \epsilon , \epsilon) \subset \mathbb{R}$. We recall the infinitesimal variation of curvature when we vary $t$, following \cite[Section 4.2]{Kobook}. Suppose that the infinitesimal variation of hermitian metrics can be written as $u := \partial_t |_{t=0} h_t $ (as an element in $\Gamma_{C^{\infty}_X} (\mathcal{E}^{\vee} \otimes \overline{\mathcal{E}^{\vee}})$), and that we write $a_t$ for the connection 1-form on $\mathcal{E}$ defined by $h_t$ and $\bar{\partial}$. Then, by fixing a holomorphic frame to use tensorial notation, we have
\begin{equation*}
	\sum_{\alpha = 1}^r (h_t)_{\alpha \bar{\gamma}} (a_t)^{\alpha}_{\beta} = \partial (h_t)_{\beta \bar{\gamma}},
\end{equation*}
for each $\beta , \gamma =1 , \dots r$, as in (1.4.10) or (4.2.9) of \cite{Kobook}. Differentiating this equation with respect to $t$, we get
\begin{equation*}
	\sum_{\alpha} u_{\alpha \bar{\gamma}} (a_0)^{\alpha}_{\beta} + \sum_{\alpha} (h_0)_{\alpha \bar{\gamma}} \partial_t |_{t=0} (a_t)^{\alpha}_{\beta} = \partial u_{\beta \bar{\gamma}}
\end{equation*}
at $t = 0$. Thus
\begin{equation}  \label{apvacnddlv}
	\sum_{\alpha} (h_0)_{\alpha \bar{\gamma}} \partial_s |_{t=0} (a_t)^{\alpha}_{\beta} = \nabla^{1,0 , \vee}_{h_0} u_{\beta \bar{\gamma}},
\end{equation}
where $\nabla^{1,0 , \vee}_{h_0}$ is the $(1,0)$-part of the covariant derivative on the dual vector bundle $\mathcal{E}^{\vee}$ defined by $h_0$ (and $\bar{\partial}$) as
\begin{equation*}
	\nabla^{1,0 , \vee}_{h_0} u_{\beta \bar{\gamma}} := \partial u_{\beta \bar{\gamma}} - \sum_{\alpha} u_{\alpha \bar{\gamma}} (a_0)^{\alpha}_{\beta}.
\end{equation*}
Hence we get
\begin{equation} \label{avconefat}
	\partial_t |_{t=0} (a_t)^{\alpha}_{\beta} = (h_0)^{\alpha \bar{\gamma}}\nabla^{1,0 , \vee}_{h_0} u_{\beta \bar{\gamma}} = \nabla^{1,0 , \vee}_{h_0} u^{\alpha}_{\beta} 
\end{equation}
by defining $u^{\alpha}_{\beta} := \sum_{\gamma} (h_0)^{\alpha \bar{\gamma}} u_{\beta \bar{\gamma}}$, as in \cite[(4.2.12)]{Kobook}, where we used $\nabla_{h_0}^{1,0, \vee} h_0^{-1} = \partial (h_0)^{\alpha \bar{\gamma}} + (a_0)^{\alpha}_{\beta} (h_0)^{\beta \bar{\gamma}} =0$. The appearance of the dual in (\ref{apvacnddlv}) and (\ref{avconefat}) can be seen e.g.~from $u \in \Gamma_{C^{\infty}_X} (\mathcal{E}^{\vee} \otimes \overline{\mathcal{E}^{\vee}})$ (see also \cite[(1.5.20)]{Kobook}).

Let $\nabla^{1,0 , \mathrm{End}}_{h_0}$ be the $(1,0)$-part of the covariant derivative on the endomorphism bundle $\mathrm{End}_{C^{\infty}_X} (\mathcal{E}) \cong \mathcal{E} \otimes \mathcal{E}^{\vee}$, defined by $\nabla^{1,0 , \mathrm{End}}_{h_0} = \partial + a_0 \otimes \Id_{\mathcal{E}^{\vee}} - \Id_{\mathcal{E}} \otimes a_0$. At each point $x \in X$ we may choose a normal coordinate system so that the connection 1-form $a_0$ of $h_0$ vanishes at $x$. With respect to this coordinate system, the equation (\ref{avconefat}) means that the variation of the curvature is given by
\begin{equation*}
	\partial_t |_{t=0} F(h_t)^{\alpha}_{\beta} = \bar{\partial} \nabla^{1,0 , \mathrm{End}}_{h_0} u^{\alpha}_{\beta}.
\end{equation*}
Since this equation is tensorial, i.e.~covariant under the change of coordinate systems, we get
\begin{equation} \label{apvacurved}
	\partial_t |_{t=0} F(h_t) = \bar{\partial} \nabla^{1,0 , \mathrm{End}}_{h_0} (h_0^{-1} \partial_t |_{t=0} h_t)
\end{equation}
irrespectively of the coordinate system chosen.

We now proceed with the proof of convexity. Along any path $\{ h_t \}_{t} \subset \mathcal{H}_{\infty}$, one has
 \begin{align*}
  \left. \frac{d^2}{d t^2} \right|_{t=0} &\mathcal{M}^{Don}(h_t,h_0) \\
  &=\int_X \tr\left(\partial_t |_{t=0} (\Lambda_\omega F_t)\cdot v_0 + (\Lambda_\omega F_0 -\mu(\mathcal{E})\Id_{\mathcal{E}})\partial_t |_{t=0} v_t \right)\frac{\omega^n}{n!}
 \end{align*}
 where $v_t :=h_t^{-1}\partial_t h_t$ is a hermitian section of $\mathrm{End}_{C^{\infty}_X} (\mathcal{E})$. If $\{ h_t \}_{t} \subset \mathcal{H}_{\infty}$ is a geodesic path, the geodesic equation $\partial_t v_t =\partial_t (h_t^{-1}\partial_t h_t)=0$ means that the second term in the above integral vanishes. Moreover, recalling (\ref{apvacurved}), we find
 \begin{align}
 \left. \frac{d^2}{d t^2} \right|_{t=0} \mathcal{M}^{Don}(h_t,h_0)&=\int_X \tr\left(\partial_t |_{t=0} ( \Lambda_\omega F_t) \cdot v_0 \right)\frac{\omega^n}{n!} \notag \\
 &=\int_X \tr( \nabla^{1,0, \mathrm{End}}_{h_0} v_0 \wedge \bar{\partial} v_0 ) \frac{\omega^{n-1}}{(n-1)!} \label{hessmdonpos} \\
 &=\Vert \nabla^{1,0,\mathrm{End}}_{h_0} v_0 \Vert^2_{L^2}\geq 0, \notag
 \end{align}
and thus $\mathcal{M}^{Don}$ is convex, since by the cocycle property of the Donaldson functional (\ref{cocyclemdon}) we may take any point of $\mathcal{H}_{\infty}$ to be the reference metric.

Suppose now that $\mathcal{E}$ is irreducible and that $\Vert \nabla^{1,0,\mathrm{End}}_{h_0} v_0 \Vert^2_{L^2} = 0$ for some smooth hermitian section $v_0 := h_t^{-1}\partial_t h_t$ of $\mathrm{End}_{C^{\infty}_X} (\mathcal{E})$ associated to a geodesic path $\{ h_t \}_t \subset \mathcal{H}_{\infty}$; recall that the geodesic equation $\partial_t (h_t^{-1}\partial_t h_t)=0$ means that $h_t^{-1}\partial_t h_t$ does not depend on $t$. This implies $\nabla^{1,0 , \mathrm{End}}_{h_0} v_0 =\bar{\partial}v_0=0$; in particular, $v_0$ is a parallel hermitian section of $\mathrm{End}_{C^{\infty}_X} (\mathcal{E})$. This means that the set of fibrewise eigenvalues of $v_0$ can be written as $\{ b_{\alpha} \}_{\alpha}$, where each $b_{\alpha}$ is a real constant since $v$ is parallel and hermitian. The subbundle $\mathcal{E}_{\alpha}$ defined by $\mathcal{E}_{\alpha} := \ker \left( v - b_{\alpha} \Id_{\mathcal{E}}\right)$ is holomorphic since $\bar{\partial} v_0 =0$, and gives the decomposition $\mathcal{E} = \bigoplus_{\alpha} \mathcal{E}_{\alpha}$, which contradicts irreducibility except for the case when $v_0$ is of the form $b \cdot \Id_{\mathcal{E}}$, $b \in \mathbb{R}$ (see also \cite[Proposition 1.1.17]{L-T1}). Thus we get $h_t = e^{bt}h_0$, i.e.~$\{ h_t \}_t$ is a trivial geodesic, thereby concluding that the Donaldson functional is strictly convex along nontrivial geodesics in $\mathcal{H}_{\infty}$.

Finally, suppose that there exist two critical points $h_0 , h_1 \in \mathcal{H}_{\infty}$ of $\mathcal{M}^{Don}$. Since $\mathcal{H}_{\infty}$ is geodesically complete, we may take a geodesic path $\{ h_t \}_t$ connecting $h_0$ and $h_1$. The convexity of $\mathcal{M}^{Don}$ implies that $h_t$ must attain the minimum for all $0 \le t \le 1$. If $\mathcal{E}$ is irreducible, the strict convexity along nontrivial geodesics that we established above implies that $\{ h_t \}_t$ must be of the form $h_t = e^{bt} h_0$ for some $b \in \mathbb{R}$, and hence $h_1 = e^b h_0$. Thus the critical point of $\mathcal{M}^{Don}$ is unique up to an overall constant scaling.
\end{proof}

 Fixing $p\geq 2$, the proof of Lemma \ref{lemmdonconvH} carries over word by word for the geodesics in $\mathcal{H}_{[p]}$ (i.e.~a path $\{ h_t \}_t \subset \mathcal{H}_{[p]}$ satisfying $\partial_t (h^{-1}_t \partial_t h_t) =0$), by replacing $\mathcal{H}_{\infty}$ by $\mathcal{H}_{[p]}$ and $C^{\infty}_X$ by $C^{p}_X$, to yield the following generalisation.

\begin{prop} \label{lemmdonconv}
 For any $p\geq 2$, the functional $\mathcal{M}^{Don}$ is convex along geodesics in $\mathcal{H}_{[p]}$, and its critical point attains the global minimum. Moreover, $\mathcal{M}^{Don}$ is strictly convex along nontrivial geodesics in $\mathcal{H}_{[p]}$ if $\mathcal{E}$ is irreducible (in particular if $\mathcal{E}$ is simple) and in this case the critical point is unique up to an overall constant scaling if it exists.
\end{prop}

\section{Quantitative $C^0$-estimate and a lower bound of the Donaldson functional}

Our aim in this section is to find a sufficient condition for the Donaldson functional to remain bounded from below, in relation to a certain quantitative $C^0$-estimate. Such a condition can be stated more precisely as follows.

\begin{definition} \label{defdelbdd}
	A hermitian metric $h \in \mathcal{H}_{\infty}$ is said to be \textbf{$\delta$-bounded} with respect to $h_0 \in \mathcal{H}_{\infty}$ if it satisfies the following: writing $\lambda_{\mathrm{max}}(x)$ (resp.~$\lambda_{\mathrm{min}}(x)$) for the largest (resp.~smallest) eigenvalue of $hh^{-1}_0$ at $x \in X$, we have
	\begin{equation*}
		\inf_{x \in X} \frac{\lambda_{\mathrm{min}}(x)}{\lambda_{\mathrm{max}}(x)} \ge \delta
	\end{equation*}
	for some fixed $0 < \delta \le 1$.
\end{definition}

\begin{rem}
	Note that the above condition is equivalent to quantitatively bounding the $C^0$-norm of $\log hh^{-1}_0$, up to fixing an overall constant multiple. We prefer the above formalism not to be bothered by the overall scaling.

In relation to bounding $\log hh^{-1}_0$, it is perhaps worth mentioning that evaluating the $L^2$-norm of $\log h_{\epsilon}h^{-1}_{\mathrm{ref}}$ for a family of hermitian metrics $\{ h_{\epsilon } \}_{0< \epsilon \ll 1}$ uniformly for all $ \epsilon >0 $ along a certain continuity path was the crucial step in the approach of Uhlenbeck--Yau \cite{U-Y}.
\end{rem}

We shall show that the Donaldson functional can be bounded from below uniformly in terms of $\delta$ for all $\delta$-bounded hermitian metrics. The proof critically relies on the convexity of the Donaldson functional, in particular on the formula in the following lemma for the second derivative of the Donaldson functional; it is almost certainly well-known to the experts (see e.g.~\cite[page 31]{Siu87}), but we provide a self-contained proof for the reader's convenience.

\begin{lem} \label{lemsecdevm}
	Let $\gamma (s)$ be a geodesic with $v:= \gamma(s)^{-1} \partial_s \gamma (s)$. Then
	\begin{equation*}
	\frac{d^2}{ds^2} \mathcal{M}^{Don} (\gamma (s) , \gamma (0) ) = \int_X \tr( e^{sv} (\nabla^{1,0,\mathrm{End}}_{\gamma(0)}v) e^{-sv} \wedge \bar{\partial} v ) \frac{\omega^{n-1}}{(n-1)!},
\end{equation*}
\end{lem}

\begin{rem}
A geometric meaning of the above formula is as follows. Since $\mathcal{M}^{Don}$ is convex along geodesics, its second derivative along a geodesic $\gamma (s)$ is always nonnegative. The above formula specifies how the second derivative changes along $\gamma (s)$, thereby quantitatively capturing the ``change in convexity'' along $\gamma (s)$.	
\end{rem}

\begin{proof}
By (\ref{hessmdonpos}), we have
$$\frac{d^2}{ds^2} \mathcal{M}^{Don} (\gamma (s) , \gamma (0) ) = \int_X \tr( \nabla_{\gamma(s)} v \wedge \bar{\partial} v ) \frac{\omega^{n-1}}{(n-1)!},$$ where $\nabla_{\gamma(s)}$ stands for (the $(1,0)$-part of) the connection $\nabla^{1,0, \mathrm{End}}_{\gamma(s)}$ on the endomorphism bundle. 

Note that
\begin{equation*}
	\nabla_{\gamma(s)} v = \nabla_{\gamma(0)} v + [v, (\nabla_{\gamma(0)}e^{sv}) e^{-sv}],
\end{equation*}
which follows from the usual transformation rule for the endomorphism bundle \cite[(1.5.16)]{Kobook}. We claim $[v,  (\nabla_{\gamma(0)}e^{sv})e^{-sv} ] = e^{sv} (\nabla_{\gamma(0)}v) e^{-sv} - \nabla_{\gamma(0)} v$. Recall first
\begin{equation*}
	\frac{d}{ds} \nabla_{\gamma(0)} e^{ s v} =  e^{ s v}(\nabla_{\gamma(0)} v ) +  (\nabla_{\gamma(0)} e^{ s v}) v,
\end{equation*}
which can be checked by the power series expansion of $e^{ s v}$. Thus
\begin{align*}
	&\frac{d}{ds} [ v, (\nabla_{\gamma(0)}e^{sv}) e^{-sv} ] \\
	&=[ v, e^{sv} (\nabla_{\gamma(0)}v) e^{-sv} ] + [v,  (\nabla_{\gamma(0)}e^{sv}) ve^{-sv} ]- [v,  (\nabla_{\gamma(0)}e^{sv})ve^{-sv}] \\
	&=e^{sv} [v, \nabla_{\gamma(0)} v] e^{-sv} \\
	&=\frac{d}{ds}e^{sv}(\nabla_{\gamma(0)}v)e^{-sv},
\end{align*}
and hence comparison at $s=0$ gives the claim. Thus
\begin{equation*}
	\frac{d^2}{ds^2} \mathcal{M}^{Don} (\gamma (s) , \gamma (0) ) = \int_X \tr( e^{sv} (\nabla_{\gamma(0)}v) e^{-sv} \wedge \bar{\partial} v ) \frac{\omega^{n-1}}{(n-1)!},
\end{equation*}
as required.

\end{proof}

We are now ready to state and prove the main result of this section.

\begin{thm}\label{propUnif}
	Suppose that $\mathcal{E}$ is irreducible and $h_0 \in \mathcal{H}_{\infty}$ be a reference metric. Let $0 < \delta \le 1$ be a fixed constant. Then there exists a constant $C(h_0)>0$ which depends only on $h_0$ such that
	\begin{equation*}
		\mathcal{M}^{Don} (h, h_0) \ge - C(h_0)  \frac{(\log \delta)^2}{\delta -1 - \log \delta}
	\end{equation*}
	for any $\delta$-bounded metric $h \in \mathcal{H}_{\infty}$ with respect to $h_0$. In particular, if $\delta >0$ is small enough, we have
	\begin{equation*}
		\mathcal{M}^{Don} (h, h_0) \ge - C(h_0) \log \delta^{-1}.
	\end{equation*}
\end{thm}

\begin{rem} \label{rprpunifqsl}
In particular, in order to establish the lower bound of $\mathcal{M}^{Don}$ on a certain path $\{ h_t \}_t \subset \mathcal{H}_{\infty}$, it suffices to show that $h_t$ is $\delta$-bounded for all $t$. This amounts to establishing the $C^0$-estimate for the continuity path, which is reminiscent of the situation for the K\"ahler--Einstein metrics on Fano manifolds (cf.~\cite[Chapter 6]{TianBook}). 
The lower bound of the renormalised Quot-scheme limit, which follows from the lower bound discussed in Remark \ref{remlbrqlimfs}, would be desirable partly because in such case we would be able to apply Theorem \ref{propUnif} to the renormalised Bergman 1-PS, thereby making a progress towards proving Hypothesis \ref{conjpczero} (see also \cite[Remark 3.10]{HK1}).
\end{rem}

\begin{proof}
We apply Lemma \ref{lemsecdevm} to the geodesic segment $\gamma(s):= e^{sv}$ ($0 \le s \le 1$), $v:= \log hh^{-1}_0$ connecting $h_0=h_0$ and $h$. We choose a local diagonalising frame for $hh^{-1}_0$ so that $hh^{-1}_0 = \mathrm{diag}(\lambda_1 , \dots , \lambda_r )$, where $\lambda_1 , \dots , \lambda_r$ are strictly positive. Then, writing $\nabla$ for $\nabla^{1,0,\mathrm{End}}_{\gamma(0)}$, we compute
\begin{equation*}
	\frac{d^2}{ds^2} \mathcal{M}^{Don} (\gamma (s) , \gamma (0) ) = \int_X \sum_{i,j=1}^r  (\lambda_i / \lambda_j)^{s} (\nabla v)_{ij} \wedge (\bar{\partial} v)_{ji}  \frac{\omega^{n-1}}{(n-1)!}
\end{equation*}
with respect to this frame. Since $h$ is $\delta$-bounded with respect to $h_0$, we have
\begin{equation*}
	\frac{d^2}{ds^2} \mathcal{M}^{Don} (\gamma (s) , \gamma (0) ) \ge  \delta^s \int_X  \tr(  (\nabla v) \wedge (\bar{\partial} v) ) \frac{\omega^{n-1}}{(n-1)!}.
\end{equation*}

Thus, integrating twice with respect to $s$ over the range $[0,1]$, we get
\begin{equation*}
	\mathcal{M}^{Don} (h , h_0) \ge C(\delta) \Vert \nabla v \Vert^2_{L^2} + \left. \frac{d}{ds} \right|_{s=0} \mathcal{M}^{Don} (\gamma (s) , h_0) ,
\end{equation*}
where $C(\delta):= (\delta -1 - \log \delta)/(\log \delta)^2$, by noting
\begin{equation*}
	\int_0^s dt \int_0^t \delta^u du = \frac{1}{\log \delta} \left( \frac{\delta^s-1}{\log \delta} -s \right) .
\end{equation*}
Note also that $C(\delta) > 0$ if $0 < \delta \le 1$.

Recall that
\begin{equation*}
	\left. \frac{d}{ds} \right|_{s=0} \mathcal{M}^{Don} (\gamma (s) , h_0) = \int_X \tr \left( v \left( \Lambda_{\omega} F(h_0) - \frac{\mu(\mathcal{E})}{\mathrm{Vol}_L} \mathrm{Id}_{\mathcal{E}} \right) \right) \frac{\omega^n}{n!}.
\end{equation*}
Now define
\begin{equation*}
	\bar{v} := \frac{1}{r \mathrm{Vol}_L} \int_X \tr (v) \frac{\omega^n}{n!} \cdot \Id_{\mathcal{E}} ,
\end{equation*}
so that $v-\bar{v}$ has average 0. Since $\bar{v}$ is a constant multiple of the identity, we have
\begin{equation*}
	\left. \frac{d}{ds} \right|_{s=0} \mathcal{M}^{Don} (\gamma (s) , h_0) = \int_X \tr \left( (v-\bar{v}) \left( \Lambda_{\omega} F(h_0) - \frac{\mu(\mathcal{E})}{\mathrm{Vol}_L} \mathrm{Id}_{\mathcal{E}} \right) \right) \frac{\omega^n}{n!}.
\end{equation*}
Thus, by using Cauchy--Schwarz, we have
\begin{equation*}
	\left. \frac{d}{ds} \right|_{s=0} \mathcal{M}^{Don} (\gamma (s) , h_0) \ge - \Vert v - \bar{v} \Vert_{L^2} \left\Vert \Lambda_{\omega} F(h_0) - \frac{\mu(\mathcal{E})}{\mathrm{Vol}_L} \Id_{\mathcal{E}} \right\Vert_{L^2}, 
\end{equation*}
where the norm $\Vert \cdot \Vert_{L^2}$ is defined by
\begin{equation*}
	\Vert v -\bar{v} \Vert^2_{L^2} := \int_X \tr \left( (v -\bar{v}) \cdot  (v -\bar{v})) \right) \frac{\omega^n}{n!}.
\end{equation*}

Thus, defining a constant
\begin{equation*}
	\underline{C}(h_0) = \left\Vert \Lambda_{\omega} F(h_0) - \frac{\mu(\mathcal{E})}{\mathrm{Vol}_L} \Id_{\mathcal{E}} \right\Vert_{L^2} \ge 0,
\end{equation*}
we see that
\begin{equation*}
	\mathcal{M}^{Don} (h,h_0) \ge C(\delta) \Vert \nabla v \Vert^2_{L^2}  -\underline{C}(h_0) \Vert v - \bar{v} \Vert_{L^2} .
\end{equation*}

Suppose now that $\mathcal{E}$ is irreducible. Then $\nabla v = \bar{\partial} v = 0$ implies that $v$ is a constant multiple of $\mathrm{Id}_{\mathcal{E}}$, as we saw in the proof of Proposition \ref{lemmdonconvH}. In particular, defining the Bochner Laplacian $\nabla^* \bar{\partial}$ by
\begin{equation*}
	\int_X \tr(\nabla v , \bar{\partial} v) \frac{\omega^n}{n!}=  \int_X \tr( v , \nabla^* \bar{\partial} v) \frac{\omega^n}{n!} ,
\end{equation*}
and writing $\mathrm{Herm}(\mathcal{E})$ for the set of all hermitian endomorphisms with average zero, we see that the first eigenvalue of $\nabla^* \bar{\partial} |_{\mathrm{Herm}(\mathcal{E})}$, which is a self-adjoint linear elliptic operator, is nonzero. (A particularly important case is when $\mathcal{E}$ is simple, i.e.~$\mathrm{End}_{\mathcal{O}_X} (\mathcal{E}) = \mathbb{C}$, where the kernel of the $\bar{\partial}$-operator on $\mathrm{End}_{C^{\infty}_X} (\mathcal{E})$ is $\mathbb{C} \cdot \mathrm{Id}_{\mathcal{E}}$.) Thus the operator $\nabla^* \bar{\partial}$ is invertible on $\mathrm{Herm}(\mathcal{E})$ and there exists a constant $C_{\nabla^* \bar{\partial}}(h_0)>0$ which depends only on $h_0$ such that
\begin{equation*}
	\Vert v- \bar{v} \Vert^2_{L^2}   \le C_{\nabla^* \bar{\partial}} (h_0) \Vert \nabla (v- \bar{v}) \Vert^2_{L^2}.
\end{equation*}

Thus we finally get, by noting $\nabla (v- \bar{v}) = \nabla v$, $C_{\nabla^* \bar{\partial}} (h_0) >0$, $C(\delta) > 0$, and by completing the square,
\begin{align*}
	\mathcal{M}^{Don} (h,h_0) &\ge C(\delta) C_{\nabla^* \bar{\partial}}(h_0)^{-1} \Vert v - \bar{v} \Vert^2_{L^2}  - \underline{C}(h_0) \Vert v - \bar{v} \Vert_{L^2},  \\
	&\ge - \frac{1}{4} C(\delta)^{-1}   \underline{C}(h_0)^2 C_{\nabla^* \bar{\partial}}(h_0) .
\end{align*}
Recalling
\begin{equation*}
	C(\delta)= \frac{\delta -1 - \log \delta}{(\log \delta)^2} \sim \frac{1}{\log \delta^{-1}}
\end{equation*}
when $\delta>0$ is small enough, we get the result.
\end{proof}

\bibliography{notes.bib}

\begin{bibdiv}
\begin{biblist}

\bib{BBS2008}{article}{
      author={Berman, Robert},
      author={Berndtsson, Bo},
      author={Sj\"ostrand, Johannes},
       title={A direct approach to {B}ergman kernel asymptotics for positive
  line bundles},
        date={2008},
        ISSN={0004-2080},
     journal={Ark. Mat.},
      volume={46},
      number={2},
       pages={197\ndash 217},
         url={https://doi.org/10.1007/s11512-008-0077-x},
      review={\MR{2430724}},
}

\bib{BBJ}{article}{
      author={Berman, Robert},
      author={Boucksom, S{\'e}bastien},
      author={Jonsson, Mattias},
       title={{A} variational approach to the {Y}au-{T}ian-{D}onaldson
  conjecture},
        date={2015},
     journal={arXiv preprint arXiv:1509.04561},
}

\bib{BHJ2}{article}{
      author={Boucksom, S{\'e}bastien},
      author={Hisamoto, Tomoyuki},
      author={Jonsson, Mattias},
       title={{U}niform {K}-stability and asymptotics of energy functionals in
  {K}\"ahler geometry},
        date={2016},
     journal={arXiv preprint arXiv:1603.01026, to appear in JEMS},
}

\bib{BHJ1}{article}{
      author={Boucksom, S{\'e}bastien},
      author={Hisamoto, Tomoyuki},
      author={Jonsson, Mattias},
       title={Uniform {K}-stability, {D}uistermaat-{H}eckman measures and
  singularities of pairs},
        date={2017},
        ISSN={0373-0956},
     journal={Ann. Inst. Fourier (Grenoble)},
      volume={67},
      number={2},
       pages={743\ndash 841},
         url={http://aif.cedram.org/item?id=AIF_2017__67_2_743_0},
      review={\MR{3669511}},
}

\bib{Catlin}{incollection}{
      author={Catlin, David},
       title={The {B}ergman kernel and a theorem of {T}ian},
        date={1999},
   booktitle={Analysis and geometry in several complex variables ({K}atata,
  1997)},
      series={Trends Math.},
   publisher={Birkh\"auser Boston, Boston, MA},
       pages={1\ndash 23},
      review={\MR{1699887}},
}

\bib{Dertwisted}{article}{
      author={Dervan, Ruadha\'\i},
       title={Uniform stability of twisted constant scalar curvature {K}\"ahler
  metrics},
        date={2016},
        ISSN={1073-7928},
     journal={Int. Math. Res. Not. IMRN},
      number={15},
       pages={4728\ndash 4783},
         url={https://doi.org/10.1093/imrn/rnv291},
      review={\MR{3564626}},
}

\bib{Do-1983}{article}{
      author={Donaldson, S.~K.},
       title={A new proof of a theorem of {N}arasimhan and {S}eshadri},
        date={1983},
        ISSN={0022-040X},
     journal={J. Differential Geom.},
      volume={18},
      number={2},
       pages={269\ndash 277},
         url={http://projecteuclid.org/euclid.jdg/1214437664},
      review={\MR{710055}},
}

\bib{Don1985}{article}{
      author={Donaldson, S.~K.},
       title={Anti self-dual {Y}ang-{M}ills connections over complex algebraic
  surfaces and stable vector bundles},
        date={1985},
        ISSN={0024-6115},
     journal={Proc. London Math. Soc. (3)},
      volume={50},
      number={1},
       pages={1\ndash 26},
         url={https://doi.org/10.1112/plms/s3-50.1.1},
      review={\MR{765366}},
}

\bib{Don1987}{article}{
      author={Donaldson, S.~K.},
       title={Infinite determinants, stable bundles and curvature},
        date={1987},
        ISSN={0012-7094},
     journal={Duke Math. J.},
      volume={54},
      number={1},
       pages={231\ndash 247},
         url={http://dx.doi.org/10.1215/S0012-7094-87-05414-7},
      review={\MR{885784}},
}

\bib{DonSun}{article}{
      author={Donaldson, Simon},
      author={Sun, Song},
       title={Gromov-{H}ausdorff limits of {K}\"{a}hler manifolds and algebraic
  geometry},
        date={2014},
        ISSN={0001-5962},
     journal={Acta Math.},
      volume={213},
      number={1},
       pages={63\ndash 106},
         url={https://doi.org/10.1007/s11511-014-0116-3},
      review={\MR{3261011}},
}

\bib{yhextremal}{article}{
      author={Hashimoto, Yoshinori},
       title={Quantisation of extremal {K}\"{a}hler metrics},
        date={2015},
     journal={arXiv preprint arXiv:1508.02643},
}

\bib{HK1}{article}{
      author={Hashimoto, Yoshinori},
      author={Keller, Julien},
       title={Quot-scheme limit of {F}ubini--{S}tudy metrics and {D}onaldson's
  functional for bundles},
        date={2018},
     journal={arXiv preprint arXiv:1809.08425},
}

\bib{H-L}{book}{
      author={Huybrechts, Daniel},
      author={Lehn, Manfred},
       title={The geometry of moduli spaces of sheaves},
     edition={Second},
      series={Cambridge Mathematical Library},
   publisher={Cambridge University Press, Cambridge},
        date={2010},
        ISBN={978-0-521-13420-0},
         url={http://dx.doi.org/10.1017/CBO9780511711985},
      review={\MR{2665168}},
}

\bib{detdiv}{article}{
      author={Knudsen, Finn~Faye},
      author={Mumford, David},
       title={The projectivity of the moduli space of stable curves. {I}.
  {P}reliminaries on ``det'' and ``{D}iv''},
        date={1976},
        ISSN={0025-5521},
     journal={Math. Scand.},
      volume={39},
      number={1},
       pages={19\ndash 55},
         url={https://doi.org/10.7146/math.scand.a-11642},
      review={\MR{0437541}},
}

\bib{Kobookjp}{inproceedings}{
      author={Kobayashi, Shoshichi},
       title={{D}ifferential {G}eometry of {H}olomorphic {V}ector {B}undles},
        date={1982},
   booktitle={{M}ath.~{S}eminar {N}otes (by {I}.~{E}noki)},
      volume={41},
   publisher={Univ. of Tokyo},
        note={in Japanese},
}

\bib{Kobook}{book}{
      author={Kobayashi, Shoshichi},
       title={Differential geometry of complex vector bundles},
      series={Publications of the Mathematical Society of Japan},
   publisher={Princeton University Press, Princeton, NJ; Princeton University
  Press, Princeton, NJ},
        date={1987},
      volume={15},
        ISBN={0-691-08467-X},
         url={http://dx.doi.org/10.1515/9781400858682},
        note={Kan\^o Memorial Lectures, 5},
      review={\MR{909698}},
}

\bib{Luebke}{article}{
      author={L\"ubke, Martin},
       title={Stability of {E}instein-{H}ermitian vector bundles},
        date={1983},
        ISSN={0025-2611},
     journal={Manuscripta Math.},
      volume={42},
      number={2-3},
       pages={245\ndash 257},
         url={https://doi.org/10.1007/BF01169586},
      review={\MR{701206}},
}

\bib{L-T1}{book}{
      author={L\"ubke, Martin},
      author={Teleman, Andrei},
       title={The {K}obayashi-{H}itchin correspondence},
   publisher={World Scientific Publishing Co., Inc., River Edge, NJ},
        date={1995},
        ISBN={981-02-2168-1},
         url={https://doi.org/10.1142/2660},
      review={\MR{1370660}},
}

\bib{M-M}{book}{
      author={Ma, Xiaonan},
      author={Marinescu, George},
       title={Holomorphic {M}orse inequalities and {B}ergman kernels},
      series={Progress in Mathematics},
   publisher={Birkh\"auser Verlag},
     address={Basel},
        date={2007},
      volume={254},
        ISBN={978-3-7643-8096-0},
      review={\MR{MR2339952 (2008g:32030)}},
}

\bib{Paul2012cm}{article}{
      author={Paul, Sean~Timothy},
       title={{CM} stability of projective varieties},
        date={2012},
     journal={arXiv preprint arXiv:1206.4923},
}

\bib{Paul13}{article}{
      author={Paul, Sean~Timothy},
       title={{S}table {P}airs and {C}oercive {E}stimates for {T}he {M}abuchi
  {F}unctional},
        date={2013},
     journal={arXiv preprint arXiv:1308.4377},
}

\bib{P-S03}{article}{
      author={Phong, D.~H.},
      author={Sturm, Jacob},
       title={Stability, energy functionals, and {K}\"ahler-{E}instein
  metrics},
        date={2003},
        ISSN={1019-8385},
     journal={Comm. Anal. Geom.},
      volume={11},
      number={3},
       pages={565\ndash 597},
         url={http://dx.doi.org/10.4310/CAG.2003.v11.n3.a6},
      review={\MR{2015757}},
}

\bib{Siu87}{book}{
      author={Siu, Yum~Tong},
       title={Lectures on {H}ermitian-{E}instein metrics for stable bundles and
  {K}\"ahler-{E}instein metrics},
      series={DMV Seminar},
   publisher={Birkh\"auser Verlag, Basel},
        date={1987},
      volume={8},
        ISBN={3-7643-1931-3},
         url={https://doi.org/10.1007/978-3-0348-7486-1},
      review={\MR{904673}},
}

\bib{Szepartial}{article}{
      author={Sz\'{e}kelyhidi, G\'{a}bor},
       title={The partial {$C^0$}-estimate along the continuity method},
        date={2016},
        ISSN={0894-0347},
     journal={J. Amer. Math. Soc.},
      volume={29},
      number={2},
       pages={537\ndash 560},
         url={https://doi.org/10.1090/jams/833},
      review={\MR{3454382}},
}

\bib{TianBook}{book}{
      author={Tian, Gang},
       title={Canonical metrics in {K}\"{a}hler geometry},
      series={Lectures in Mathematics ETH Z\"{u}rich},
   publisher={Birkh\"{a}user Verlag, Basel},
        date={2000},
        ISBN={3-7643-6194-8},
         url={https://doi.org/10.1007/978-3-0348-8389-4},
        note={Notes taken by Meike Akveld},
      review={\MR{1787650}},
}

\bib{U-Y}{article}{
      author={Uhlenbeck, K.},
      author={Yau, S.-T.},
       title={On the existence of {H}ermitian-{Y}ang-{M}ills connections in
  stable vector bundles},
        date={1986},
        ISSN={0010-3640},
     journal={Comm. Pure Appl. Math.},
      volume={39},
      number={S, suppl.},
       pages={S257\ndash S293},
         url={https://doi.org/10.1002/cpa.3160390714},
        note={Frontiers of the mathematical sciences: 1985 (New York, 1985)},
      review={\MR{861491}},
}

\bib{U-Y2}{article}{
      author={Uhlenbeck, K.},
      author={Yau, S.-T.},
       title={A note on our previous paper: ``{O}n the existence of
  {H}ermitian-{Y}ang-{M}ills connections in stable vector bundles'' [{C}omm.\
  {P}ure {A}ppl.\ {M}ath.\ {\bf 39} (1986), {S}257--{S}293; {MR}0861491
  (88i:58154)]},
        date={1989},
        ISSN={0010-3640},
     journal={Comm. Pure Appl. Math.},
      volume={42},
      number={5},
       pages={703\ndash 707},
         url={https://doi.org/10.1002/cpa.3160420505},
      review={\MR{997570}},
}

\bib{Wang1}{article}{
      author={Wang, Xiaowei},
       title={Balance point and stability of vector bundles over a projective
  manifold},
        date={2002},
        ISSN={1073-2780},
     journal={Math. Res. Lett.},
      volume={9},
      number={2-3},
       pages={393\ndash 411},
         url={http://dx.doi.org/10.4310/MRL.2002.v9.n3.a12},
      review={\MR{1909652}},
}

\bib{Wang2}{article}{
      author={Wang, Xiaowei},
       title={Canonical metrics on stable vector bundles},
        date={2005},
        ISSN={1019-8385},
     journal={Comm. Anal. Geom.},
      volume={13},
      number={2},
       pages={253\ndash 285},
         url={http://projecteuclid.org/euclid.cag/1117656986},
      review={\MR{2154820}},
}

\end{biblist}
\end{bibdiv}

\vspace{2cm}
\address{{\small
\noindent\begin{tabular}{rp{7.8cm}}
 {\sc Yoshinori Hashimoto} & {\sc   Department of Mathematics, Tokyo Institute of Technology, 2-12-1 Ookayama, Meguro-ku, Tokyo, 152-8551, Japan.}\\
{\it Email address: } &\email{hashimoto@math.titech.ac.jp}\\
                      &  \\
{\sc Julien Keller} &  {\sc Aix Marseille Universit\'e, CNRS, Centrale Marseille, Institut de Math\'ematiques
de Marseille, UMR 7373, 13453 Marseille, France.}\\
{\it Email address: }& \email{julien.keller@univ-amu.fr}\\
\end{tabular}}}

\end{document}